\documentclass[reqno,11pt]{amsart}

\usepackage{amsmath,amsfonts,amssymb,amsthm,epsfig,amscd,}
\usepackage{color,comment}
\usepackage[]{fontenc}
\usepackage[latin1]{inputenc}
\usepackage{enumerate}
\usepackage[cmtip,all]{xy}
\usepackage{tabularx,multirow}
 \allowdisplaybreaks \numberwithin{equation}{section}

\newtheorem{theorem}{Theorem}[section]
\newtheorem{proposition}[theorem]{Proposition}
\newtheorem{corollary}[theorem]{Corollary}
\newtheorem{lemma}[theorem]{Lemma}
\newtheorem{conjecture}[theorem]{Conjecture}

\newtheorem{claim}[theorem]{Claim}
\newtheorem*{bdelu}{Theorem}

\theoremstyle{definition}
\newtheorem{definition}[theorem]{Definition}

\newtheorem{question}[theorem]{Question}

\theoremstyle{remark}
\newtheorem{remark}[theorem]{Remark}

\DeclareMathOperator{\ord}{ord}
\DeclareMathOperator{\conngon}{conn.gon}
\DeclareMathOperator{\covgon}{cov.gon}
\DeclareMathOperator{\gon}{gon}
\DeclareMathOperator{\codim}{codim}
\DeclareMathOperator{\Sym}{Sym}

\DeclareMathOperator{\Pic}{Pic}

\def\bP{{\mathbb P}}
\def\bG{{\mathbb G}}
\def\cA{{\mathcal A}}
\def\cB{{\mathcal B}}
\def\fM{{\mathcal M}}
\def\fN{{\mathcal N}}
\def\cE{{\mathcal E}}
\def\cL{{\mathcal L}}
\def\cI{{\mathcal I}}
\def\cO{{\mathcal O}}
\def\cP{{\mathcal P}}
\def\cX{{\mathcal X}}
\def\cY{{\mathcal Y}}

\def\ddt{\widetilde{\Delta}}

\begin{document}

\title[Gonality of curves on general hypersurfaces]{Gonality of curves on general hypersurfaces}

\author{Francesco Bastianelli}
\address{Francesco Bastianelli, Dipartimento  di Matematica, Universit\`{a} degli Studi di Bari ``Aldo Moro", Via Edoardo Orabona 4, 70125 Bari -- Italy}
\email{francesco.bastianelli@uniba.it}
\thanks{This work was partially supported by MIUR FIRB 2012 \emph{``Spazi di moduli e applicazioni''}; MIUR PRIN 2010--2011  \emph{``Geometria delle variet\`a algebriche''}; INdAM (GNSAGA); and by the research project \emph{``Families of curves: their moduli and their related varieties"} (CUP: E81-18000100005) - Mission Sustainability - University of Rome Tor Vergata.\\
C. Ciliberto and F. Flamini acknowledge the MIUR Excellence Department Project awarded to the Department of Mathematics, University of Rome Tor Vergata (CUP: E83C18000100006).}

\author{Ciro Ciliberto}
\address{Ciro Ciliberto, Dipartimento  di Matematica, Universit\`{a} degli Studi di Roma ``Tor Vergata", Viale della Ricerca Scientifica 1, 00133 Roma -- Italy}
\email{cilibert@mat.uniroma2.it}

\author{Flaminio Flamini}
\address{Flaminio Flamini, Dipartimento  di Matematica, Universit\`{a} degli Studi di Roma ``Tor Vergata", Viale della Ricerca Scientifica 1, 00133 Roma -- Italy}
\email{flamini@mat.uniroma2.it}

\author{Paola Supino}
\address{Paola Supino, Dipartimento  di Matematica e Fisica, Universit\`{a} degli Studi ``Roma Tre", Largo S. L. Murialdo 1, 00146 Roma -- Italy}
\email{supino@mat.uniroma3.it}

\begin{abstract}
This paper concerns the existence of curves with low gonality on smooth hypersurfaces of sufficiently large degree.
It has been recently proved that if $X\subset \mathbb{P}^{n+1}$ is a hypersurface of degree $d\geqslant n+2$, and if  $C\subset X$ is an irreducible curve passing through a general point of $X$, then its gonality verifies $\gon(C)\geqslant d-n$, and equality is attained on some special hypersurfaces.
We prove that if $X\subset \mathbb{P}^{n+1}$ is a very general hypersurface of degree $d\geqslant 2n+2$, the least gonality of an irreducible curve $C\subset X$ passing through a general point of $X$ is $\gon(C)=d-\left\lfloor\frac{\sqrt{16n+1}-1}{2}\right\rfloor$, apart from a series of possible exceptions, where $\gon(C)$ may drop by one.
\end{abstract}

\maketitle

\section{Introduction}

In this paper, we consider smooth hypersurfaces $X\subset \mathbb{P}^{n+1}$ of  sufficiently large degree, and we are interested in the existence of irreducible curves $C\subset X$ having low gonality and passing through a general point of $X$. 
We recall that the \emph{gonality} $\gon(C)$ of an irreducible projective curve $C$ is the least degree of a non--constant morphism $\widetilde{C}\longrightarrow \mathbb{P}^1$, where $\widetilde{C}$ is the normalization of $C$.

The study of curves on varieties is at the basis of the birational classification.
As for rational curves on projective hypersurfaces or complete intersections, their existence has been investigated long since (see e.g. \cite{Morin,Pre,Segre}), and it has been understood in a series of seminal works (see \cite{Cle,E,E2,V,V2}).
In particular, it turns out that no rational curve lies on a very general hypersurface $X\subset \mathbb{P}^{n+1}$ of degree $d\geqslant 2n$.
When instead $X$ has degree $d\leqslant 2n-1$, it contains lines varying in a $(2n-d-1)$--dimensional family (see e.g. \cite{Bo,DeMa}), and if $n\geqslant 5$,  lines are the only rational curves on a very general hypersurface of degree $d=2n-1$ (cf. \cite{P}).   

\smallskip
In order to deal with (possibly moving) curves having higher gonality, it is profitable to consider the following birational invariant introduced in \cite{BDELU}, and coinciding with the gonality in dimension 1.
Given an irreducible variety $Y$, we define the \emph{covering gonality} of $Y$ to be the integer
\begin{displaymath}
\covgon(Y):=\min\left\{c\in \mathbb{N}\left|
\begin{array}{l}
\text{Given a general point }y\in Y,\,\exists\text{ an irreducible}\\ \text{curve } C\subseteq Y  \text{ such that }y\in C \text{ and }\gon(C)=c
\end{array}\right.\right\}.
\end{displaymath}
Since $\covgon(Y)=1$ is equivalent to $Y$ being uniruled, we can think of the covering gonality as a measure of the failure of $Y$ to be uniruled.

{The following theorem is probably the most general result governing the gonality of moving curves in a very general hypersurface of large degree.
\begin{bdelu}[{\cite[Proposition 3.8]{BDELU}}]
Let $X\subset \mathbb{P}^{n+1}$ be a very general hypersurface of degree $d\geqslant 2n$.
If $Y\subset X$ is an irreducible subvariety of dimension $s\geqslant 1$, one has
\begin{equation}\label{eq:proposition3.8}
\covgon(Y)\geqslant d-2n+s.
\end{equation}
In particular, 
\begin{equation}\label{eq:theoremA}
\covgon(X)\geqslant d-n.
\end{equation}
\end{bdelu}}

Actually \eqref{eq:theoremA}  holds  even if  $X$ has at worst canonical singularities and $d\geqslant n+2$; under these assumptions, the bound \eqref{eq:theoremA} is  sharp (cf. \cite[Corollary 1.11 and Example 1.7]{BDELU}).

When $n=1$ and $X\subset \mathbb{P}^{2}$ is a smooth plane curve of degree  $d \geqslant 3$, a famous result by M. Noether yields $\gon(X)=d-1$, and all the morphisms $X\longrightarrow \mathbb{P}^1$ of degree $d-1$ are projections from a  point of $X$ (see e.g. \cite{C2}).  
 The case of smooth surfaces $X\subset\mathbb{P}^3$ of degree $d\geqslant 5$ has been studied in \cite{LP}, where one shows that $\covgon(X)=d-2$, and all families of curves computing the covering gonality are classified.
In particular, if $X\subset \mathbb{P}^3$ is a very general surface of degree $d\geqslant 5$ and $C\subset X$ is a $(d-2)$-gonal curve passing through a general point $x\in X$, then $C$ is a plane curve cut out on $X$ by some tangent plane $T_pX$, so that $C$ has a double point at $p\in X$ and the map $C\dashrightarrow \mathbb{P}^1$ of degree $d-2$ is the projection from $p$ (cf. \cite[Corollary 1.8]{LP}).

\smallskip
The main result in this paper is the following Theorem \ref {theorem:Main}, which determines the covering gonality of a very general hypersurface $X\subset \mathbb{P}^{n+1}$ of  sufficiently large degree and arbitrary dimension, apart from a series of exceptions for which, as we will see, the covering gonality is \emph{almost} determined (see Remark \ref {rem:almost} below).

\begin{theorem}\label{theorem:Main}
Let $X\subset \mathbb{P}^{n+1}$ be a very general hypersurface of degree $d\geqslant 2n+2$.
Then 
\begin{equation}
\label{eq:CovGon}
d-\left\lfloor\frac{\sqrt{16n+9}-1}{2}\right\rfloor\leqslant \covgon(X)\leqslant d-\left\lfloor\frac{\sqrt{16n+1}-1}{2}\right\rfloor.
\end{equation}
If  ${n\in\mathbb{N}\smallsetminus\left\{\left.4\alpha^2+3\alpha,4\alpha^2+5\alpha+1\right|\alpha\in \mathbb{N}^*\right\}}$, then
\begin{equation}\label{eq:CovGonEq}
\covgon(X)= d-\left\lfloor\frac{\sqrt{16n+1}-1}{2}\right\rfloor,
\end{equation}
and for general $x\in X$, there exists an irreducible plane curve $C\subset X$ passing through $x$, which computes the covering gonality via the projection $C\dashrightarrow \mathbb{P}^1$ from a singular point $p\in C$ of multiplicity $\left\lfloor\frac{\sqrt{16n+1}-1}{2}\right\rfloor$.
\end{theorem}

\begin{remark}\label {rem:almost}
For ${n\in\left\{\left.4\alpha^2+3\alpha,4\alpha^2+5\alpha+1\right|\alpha\in \mathbb{N}\right\}}$, the integers $\left\lfloor\frac{\sqrt{16n+9}-1}{2}\right\rfloor$ and $\left\lfloor\frac{\sqrt{16n+1}-1}{2}\right\rfloor$ differ by 1, so \eqref{eq:CovGon}  almost determines the covering gonality (cf. Lemma \ref{lemma:Equality}).\end{remark} 

\smallskip
To prove Theorem \ref{theorem:Main}, we combine various approaches and techniques developed in several works (see \cite{BCD,BDELU,C,E,LP,P,P2,V}). 
The key idea is to relate irreducible curves of low gonality contained in $X\subset \mathbb{P}^{n+1}$ to the cones $V_p^h\subset \mathbb{P}^{n+1}$ swept out by tangent lines having intersection multiplicity at least $h\geqslant 2$ at $p\in X$. 

To this aim, we use the argument of \cite{B,LP}, and we deduce from \cite[Theorem 2.5]{BCD} that, if $C\subset X$ is an irreducible curve which passes through a general point $x\in X$ and admits a map 
$$\varphi\colon C\dashrightarrow \mathbb{P}^1$$ 
of small degree $c\leqslant d-3$, then any fiber of $\varphi$ consists of collinear points (see Proposition \ref{proposition:BCD}).    
By arguing as in \cite[Theorem C]{BDELU}, we prove further that a curve $C\subset X$ as above lies on a cone $V_p^{d-c}$, and the map $\varphi\colon C\dashrightarrow \mathbb{P}^1$ is the projection from the vertex $p\in C$ (cf. Proposition \ref{proposition:Cone}).

Then we follow the approach of \cite{P}, relying on vector bundles techniques as in \cite{E,V}. 
We consider the Grassmannian $\mathbb{G}(1,n+1)$ of lines $\ell\subset \mathbb{P}^{n+1}$ and we define the locus
$$
\ddt_{d-c,X}:=\left\{\left.\left(p,\left[\ell\right]\right)\in X\times\mathbb{G}(1,n+1) \right|\ell\cdot X\geqslant (d-c)p\right\}.
$$
The gonality map $\varphi\colon C\dashrightarrow \mathbb{P}^1$  clearly determines a rational curve in $\ddt_{d-c,X}$.  
Since the curves $C$ cover  $X$, then $\ddt_{d-c,X}$ contains a uniruled subvariety of dimension at least $n$.
This, and an analysis of positivity properties of the canonical bundle of $\ddt_{d-c,X}$, yield numerical restrictions on $d-c$, leading to the lower bound in Theorem \ref{theorem:Main} (cf. Corollary \ref{corollary:RationalCurvesVeryGen} and Theorem \ref{theorem:LowerBound}).

In order to conclude the proof of Theorem \ref{theorem:Main}, we  construct a family of irreducible plane curves covering $X$ and having a singularity of multiplicity $\left\lfloor\frac{\sqrt{16n+1}-1}{2}\right\rfloor$. 
The cones $V_p^h$ are contained in the tangent hyperplane $T_pX\cong \mathbb{P}^n$, and any hyperplane section of $V_p^h$ not containing $p$ is defined by the vanishing of $h-2$ polynomials of degrees $2,3,\dots,h-1$, respectively. 
Then we slightly improve (in the case of lines) a classical result about linear spaces in complete intersections in a projective space (cf. \cite{Pre,C} and Proposition \ref{proposition:Predonzan+}), from which we deduce the existence of a subvariety $Z\subset X$ of dimension at least $n-1$ such that for any $p\in Z$ and $h\leqslant \left\lfloor\frac{\sqrt{16n+1}-1}{2}\right\rfloor$, the cone $V_p^h$ contains a line $\ell_p$ not passing through $p$ (see Lemma \ref{lemma:Z}).
Hence the span of $p$ and $\ell_p$ cuts out on $X$ a plane curve $C_p$ having a singularity at $p$ of multiplicity at least $h$, so that the projection from $p\in C_p$ is a map $C_p\dashrightarrow \mathbb{P}^1$ of degree at most $d-h$.
The family of plane curves obtained by varying $p\in Z$  covers $X$, and the assertion follows by setting $h= \left\lfloor\frac{\sqrt{16n+1}-1}{2}\right\rfloor$ (see Theorem \ref{theorem:UpperBound}).

\smallskip
A couple of questions are in order. First, it would be interesting to characterize the curves computing the covering gonality of $X$  and, in particular, to understand whether, at least if $n$ is sufficiently large and $d\geqslant 2n+1$, they are only the plane curves presented above. 
We discuss this question in \S \ref{ssec:plane}. 

Concerning the exceptional values ${n\in\left\{\left.4\alpha^2+3\alpha,4\alpha^2+5\alpha+1\right|\alpha\in \mathbb{N}\right\}}$,  apart from the trivial case $n=1$, we cannot decide if \eqref{eq:CovGonEq} holds for other  exceptional values. 
However, especially if one believes that for sufficiently large dimension of $X$ the covering gonality is computed by plane curves, it is natural to make the following:
\begin{conjecture}
Let $X\subset \mathbb{P}^{n+1}$ be a very general hypersurface of degree $d\geqslant 2n$.
Then 
\begin{equation*}
\covgon(X)= d-\left\lfloor\frac{\sqrt{16n+1}-1}{2}\right\rfloor.
\end{equation*}
\end{conjecture}

\smallskip
The paper is organized as follows.
In Section \ref{Section:CoveringFamilies} we are concerned with the relations between the cones of lines having high tangency order at a point $p\in X$ and the geometry of curves having low gonality and covering $X$.
On one hand, we discuss the existence of lines in $V_p^h$ not belonging to the ruling of the cone, and we achieve the upper bound in Theorem \ref{theorem:Main}.
On the other hand, we prove that any curve $C\subset X$ through a general point of $X$ having sufficiently small gonality lies on some $V_p^h$ and we describe the  gonality map $C\dashrightarrow \mathbb{P}^1$.

In Section \ref{section:UniversalFamilies} we study positivity properties of the loci $\ddt_{d-c,X}$, deducing numerical conditions on the existence of uniruled subvarieties of $\ddt_{d-c,X}$.  
In Section \ref{section:Proof} we finish the proof of Theorem \ref{theorem:Main}, and in Section \ref {sec:spec} we make some final remarks and discuss some open problems.

\subsection*{Notation}

We work over $\mathbb{C}$.
By \emph{variety} we mean a complete reduced algebraic variety $X$, unless otherwise stated.
By \emph{curve} we mean a variety of dimension 1.
We say that a property holds for a \emph{general} (resp. \emph{very general}) point ${x\in X}$ if it holds on a Zariski open nonempty subset of $X$ (resp. on the complement of the countable union of proper subvarieties of $X$).


\section{Geometry of covering families of curves}\label{Section:CoveringFamilies}

\subsection{High tangency cones to hypersurfaces}\label{section:cones}

Let $X\subset \mathbb{P}^{n+1}$ be a hypersurface defined by the vanishing of a non--zero homogeneous polynomial $F\in \mathbb{C}[y_0,\ldots,y_{n+1}]$ of degree $d\geqslant 2$.
\begin{definition}
Given a point $x\in X$ and an integer $2\leqslant h\leqslant d$, the \emph{cone} $V^h_x(X)\subset \mathbb{P}^{n+1}$ \emph{of tangent lines of order} $h$ \emph{at} $x\in X$ (denoted by $V^ h_x$ if there is no danger of confusion) is the set of all lines $\ell\subset \mathbb{P}^{n+1}$ having intersection multiplicity at least $h$ with $X$ at $x$.
\end{definition}
The variety $V^h_x$ is a cone  with vertex containing $x$, and it is defined by the $h-1$ equations deduced from the Taylor expansion of $F$ at the point $x$,
\begin{equation}\label{eq:Gj}
G_k(y_0,\ldots,y_{n+1}):=\sum_{1\leqslant i_1\leqslant \dots\leqslant i_k\leqslant n+1}y_{i_1}\cdots y_{i_k} \frac{\partial^k F}{\partial y_{i_1}\cdots \partial y_{i_k}}(x)=0 \quad \text{for } 1\leqslant k\leqslant  h-1.
\end{equation}
In particular, if $X$ is smooth at $x$, then $V^2_x$ is the (projective) tangent hyperplane $T_xX\subset \mathbb{P}^{n+1}$. 
When $h\geqslant 3$, the variety $V^h_x$ is  a cone in $T_xX\cong \mathbb{P}^{n}$ with vertex at $x\in X$, and any hyperplane section of $V_x^h$ not containing $x$ is a subvariety $\Lambda^h_{x}\subset \mathbb{P}^{n-1}$, uniquely determined up to isomorphism, defined by the vanishing of $h-2$ polynomials of degrees $2,3,\dots,h-1$, respectively.

\begin{lemma}\label{lemma:LambdaGeneral}
Let $X\subset \mathbb{P}^{n+1}$ be a general hypersurface of degree $d\geqslant 2$, let $x\in X$ be a general point, and let $2\leqslant h\leqslant d$ be an integer.
Then $\Lambda_{x}^h$ is a general complete intersection of type $(2,3,\dots,h-1)$ in $\mathbb{P}^{n-1}$. 
\begin{proof}
We may assume that $x=[0,\ldots,0,1]$ and that $T_{x}X=V(y_{n})$.
Then $X$ has equation of the form
$$
F(y_0,\ldots, y_{n+1})=y_ny_{n+1}^{d-1}+f_2(y_0,\ldots,y_n)y_{n+1}^{d-2}+\ldots+f_{d-1}(y_0,\ldots,y_n)y_{n+1}+f_d(y_0,\ldots,y_n)=0,
$$
where the polynomials $f_i$ are homogeneous of degree $i$. 
Then $\Lambda^h_{x}$, as a subvariety of the $\mathbb P^ {n-1}$ with equations $y_n= y_{n+1}=0$, is defined by the equations
$$
f_2(y_0,\ldots,y_{n-1},0)=\dots=f_{h-1}(y_0,\ldots,y_{n-1},0)=0.
$$
Let $W$ be  the  sub--vector space of $\mathbb{C}[y_0,\ldots,y_{n+1}]_d$ consisting of all polynomials $F$ as above.
Consider the linear map
\begin{equation}\label{eq:zeta}
\begin{array}{rrcl}
  \zeta\colon & W  &  \longrightarrow & \displaystyle \prod_{i=2}^{h-1} H^{0}\left(\mathbb{P}^{n-1},\mathcal{O}_{\mathbb{P}^{n-1}}(i) \right)  \\
     &  F & \longmapsto & \displaystyle \Big(f_2(y_0,\ldots,y_{n-1},0),\ldots, f_{h-1}(y_0,\ldots,y_{n-1},0)\Big).
\end{array}
\end{equation}
To prove the assertion amounts to show that  $\zeta$ is  surjective.
One has
$$
\dim (W)= {d+n+1\choose d}-n-2\quad \text{and}\quad \dim \left(\prod_{i=2}^{h-1} H^{0}\left(\mathbb{P}^{n-1},\mathcal{O}_{\mathbb{P}^{n-1}}(i) \right)\right)={h+n-1\choose h-1}-n-1. 
$$ 
Moreover ${\rm ker}(\zeta)$ consists of polynomials divisible by $y_n$, 
i.e. polynomials $F$ such that $f_i(y_0,\ldots,y_n)=y_ng_{i-1}(y_0,\ldots,y_n)$, where $2\leqslant i\leqslant h-1$ and the polynomials $g_{j}$ have degree $j$.
Then 
$$
\dim ({\rm ker}(\zeta))={d+n+1\choose d}-{h+n-1\choose h-1}-1= \dim (W)- \dim \left(\prod_{i=2}^{h-1} H^{0}\left(\mathbb{P}^{n-1},\mathcal{O}_{\mathbb{P}^{n-1}}(i) \right)\right),
$$
proving that  $\zeta$ is surjective.
\end{proof}
\end{lemma}


\subsection{Lines on complete intersections}\label{section:Predonzan+}

Let $1\leqslant s \leqslant m-2$ and $d_1,d_2,\dots,d_s\in \mathbb{N}$ be integers, and set $\underline{d} := (d_1,\ldots,d_s)$. 
We consider the vector space
$$
S_{\underline{d}}:= \bigoplus_{i=1}^s H^{0}\left(\mathbb P^ m,\mathcal{O}_{\mathbb{P}^{m}}(d_i) \right),
$$
and its Zariski open subset
$$
S^*_{\underline{d}}:= \bigoplus_{i=1}^s \left( H^{0}\left(\mathbb P^ m,\mathcal{O}_{\mathbb{P}^{m}}(d_i) \right)\setminus \{0\}\right).
$$
For any $u := (F_1,\ldots,F_s) \in S^*_{\underline{d}}$, we denote by $Y_u:=V (F_1,\ldots,F_s)\subset \mathbb{P}^m$ the closed subscheme defined by the vanishing of the $s$ homogeneous polynomials $F_1,\ldots,F_s$.
When $u \in S^*_{\underline{d}}$ is a general point, $Y_u\subset \mathbb{P}^m$ is a smooth irreducible complete intersection of dimension $m-s \geqslant 2$.
When $s=1$, we set $d=d_1$ and we simply write $S_{d}:= H^{0}\left(\mathbb P^ m,\mathcal{O}_{\mathbb{P}^{m}}(d) \right)$ and $S^*_{d}:= H^{0}\left(\mathbb P^ m,\mathcal{O}_{\mathbb{P}^{m}}(d) \right)\setminus \{0\}$.
Moreover, given a polynomial $F\in S^*_d$, we denote by  $X_F:=V(F)\subset  \mathbb{P}^m$ the hypersurface defined by the vanishing of $F$.

In \cite{Pre}, Predonzan gave necessary and sufficient conditions for the existence of a $k$--dimensional linear subvariety of $Y_u$, with $u := (F_1,\ldots,F_s) \in S^*_{\underline{d}}$ (see also \cite[Theorem 2.1]{C} and \cite{Bo, DeMa}).
The following proposition provides a slight extension of Predonzan's result in the case $k=1$.
We would like to mention that, after having proved this statement, we noticed that related results concerning the cases $k\geqslant 1$ are obtained in \cite{Mi} via a different strategy. 

\begin{proposition}\label{proposition:Predonzan+}
Let $1 \leqslant s \leqslant m-2$  and $d_1,d_2,\dots,d_s$ be positive integers such that $\Pi_{i=1}^s d_i >2$.
Consider the locus
$$
W_{\underline{d}}:=\left\{u\in S^*_{\underline{d}} \left| Y_u \text{ contains a line}\right.\right\} \subseteq S_{\underline{d}}
$$
and set 
\begin{equation}\label{eq:inequality}
t:= \max \left\{ 0, \sum_{i=1}^s d_i + s - 2(m-1)\right\} \quad \text{and}\quad \theta:= \max \left\{ 0, 2(m-1)-\sum_{i=1}^s d_i - s \right\}.
\end{equation}
Then $W_{\underline{d}}$ is nonempty, irreducible, and
$$
\codim_{S_{\underline{d}}}(W_{\underline{d}}) = t.
$$
Furthermore, if $u \in W_{\underline{d}}$ is a general point, then $Y_u \subset \mathbb{P}^m$ is a smooth complete intersection of dimension $m-s$, containing a family of lines of dimension $\theta$. 
If $t>0$, and if $u \in W_{\underline{d}}$ is a general point, then $Y_u$  contains a unique line.
\begin{proof}  
The case $t=0$, is the one considered by Predonzan, in which the assertion holds
$W_{\underline{d}}=S^*_{\underline{d}}$ (cf. \cite[Theorem 2.1]{C}).

Hence we may assume $t>0$. 
Consider the incidence correspondence
$$
J:=\left\{\left.\left(\left[\ell\right], u \right) \in \mathbb{G}(1,m) \times S^*_{\underline{d}}\right| \ell  \subset Y_u\right\}
$$ 
with the projections
$$\mathbb{G}(1,m) \stackrel{\pi_1}{\longleftarrow} J \stackrel{\pi_2}{\longrightarrow} S^*_{\underline{d}}.$$
Notice that $J$ is an open dense subset of a vector bundle over $\mathbb{G}(1,m)$ via $\pi_1$. 
Indeed, for any $[\ell]\in \mathbb{G}(1,m)$, the fiber $\pi_1^{-1}\left(\left[\ell\right]\right)$  equals 
$\bigoplus_{i=1}^s \big ( H^{0}\left(\mathcal{I}_{\ell/\mathbb{P}^{m}}(d_i)\right)\setminus \{0\}\big )$, where $\mathcal{I}_{\ell/\mathbb{P}^{m}}$ is the ideal sheaf of $\ell$ in $\mathbb{P}^{m}$. 
Thus $J$ is smooth and irreducible and
$$
 \dim (J) = 2(m-1) + \sum_{i=1}^s h^{0}\left(\mathcal{O}_{\mathbb{P}^{m}}(d_i)\right) -\sum_{i=1}^s (d_i+1)= \sum_{i=1}^s {d_i+m\choose m} - t>0,
$$
where we used the assumption $t>0$. Since $W_{\underline{d}}=\pi_2(J)$, then $W_{\underline{d}}$ is nonempty and irreducible.

We claim that, for any $[\ell]\in \mathbb{G}(1,m)$, if 
$([\ell],u)\in \pi_1^{-1}\left(\left[\ell\right]\right)$ is general (so that $u\in W_{\underline{d}}$ is general), then $Y_u$ is a smooth complete intersection of dimension $m-s$. This is an immediate consequence of Bertini's theorem applied to the blow--up of $\mathbb P^ m$ along $\ell$, noting that the strict transforms of the linear systems $|H^{0}\left(\mathcal{I}_{\ell/\mathbb{P}^{m}}(d_i)\right)|$, with $1\leqslant i\leqslant s$, are base point free.

 If $u\in W_{\underline{d}}$ is general, then 
$$
 \dim (W_{\underline{d}})  =  \dim (J)  -  \dim (\pi_2^{-1}(u)) = \sum_{i=1}^s {d_i+m\choose m} - t -  \dim (\pi_2^{-1}(u)),
 $$
 thus \eqref{eq:inequality} gives 
\begin{equation}
\begin{split}\label{codimT}
 \codim_{S_{\underline{d}}} (W_{\underline{d}}) &  = t + \dim (\pi_2^{-1}(u)).
\end{split}
\end{equation}

Next we show that, for $u\in W_{\underline{d}}$ general, one has $\dim\left(\pi_2^{-1}(u)\right)=0$.
To this aim, we  argue as in \cite[Proposition 2.1]{Bo} and therefore we will be brief.
Let $[\ell] \in \mathbb{G}(1,m)$ and let $[y_0,y_1, \ldots, y_m]$ be coordinates in $\mathbb{P}^m$ such that $I_{\ell} := (y_2,\dots,y_m)$. 
For 
$$
\left(\left[\ell\right],u\right) \in \pi_1^{-1}\left(\left[\ell\right]\right)\subset J,\quad \text{with } \quad u = (F_1, \ldots, F_s),
$$
we can write
$$
F_h = \sum_{i=2}^m y_i \; P_h^{(i)} + R_h, \; 1 \leqslant h \leqslant s,
$$
where
\begin{equation}\label{eq:Pt}
P_h^{(i)} = \sum_{\mu_0+\mu_1 = d_h-1} c^{(i)}_{h,\mu_0,\mu_1} y_0^{\mu_0} y_1^{\mu_1} \in \mathbb{C}[y_0,y_1]_{d_h-1}, \quad \text{for } 1 \leqslant h \leqslant s \text{ and }  2 \leqslant i \leqslant m,
\end{equation} 
whereas $R_h \in I^2_{\ell}$.  
We may assume $u$ general, so that $Y_u$ is smooth and the normal sheaf $N_{\ell/Y_{u}}$ is a vector bundle on $\ell$, fitting in the exact sequence
\begin{equation}\label{eq:normalbundle}
0 \to N_{\ell/Y_u} \to N_{\ell/\mathbb{P}^m} \cong \mathcal{O}_{\mathbb{P}^{1}}(1)^{\oplus (m-1)} \to \left.N_{Y_u/\mathbb{P}^m}\right|_{\ell} \cong \bigoplus_{h=1}^s\mathcal{O}_{\mathbb{P}^{1}}(d_h)  \to 0.
\end{equation}
Any $\xi \in H^0(\ell, N_{\ell/\mathbb{P}^m})$ can be identified with a collection of $m-1$ linear forms on $\mathbb{P}^1 \cong \ell$
$$
\varphi_i^{\xi}(y_0,y_1) := a_{i,0} y_0 + a_{i,1} y_1, \quad\text{with } 2 \leqslant i \leqslant m,
$$
whose coefficients form the $(m-1) \times 2$ matrix
$$
A_\xi:= (a_{i,j}),\quad\text{where } 2 \leqslant i \leqslant m\text{ and } 0 \leqslant j \leqslant 1.
$$
By abusing notation, we identify $\xi$ with $A_{\xi}$. 
Then the map  
\[H^0\left(\ell, N_{\ell/\mathbb{P}^m}\right) \stackrel{\sigma}{\longrightarrow} H^0\left(\ell, \left.N_{Y_{u_{\ell}}/\mathbb{P}^m}\right|_{\ell}\right)\]
 arising from \eqref{eq:normalbundle},  is given by
\begin{equation}\label{eq:linsys}
A_{\xi} \stackrel{\sigma}{\longrightarrow} \left(\sum_{0 \leqslant j \leqslant 1 < i \leqslant m} a_{i,j} y_j P_h^{(i)} \right)_{1 \leqslant h \leqslant s}.
\end{equation}
Notice that  $t>0$  is equivalent to $h^0\left(\ell, N_{\ell/\mathbb{P}^m}\right) < h^0\left(\ell, \left.N_{Y_{u}/\mathbb{P}^m}\right|_{\ell}\right)$. 

\begin{claim}\label{claim:2} The map $\sigma$ is injective.
Thus $h^0(N_{\ell/{Y_{u}}}) =0$, i.e. the {\em Fano scheme} $F({Y_u})$ parametrizing lines in ${Y_u}$ contains $[\ell]$ as a zero--dimensional integral component.
\begin{proof}[Proof of Claim \ref{claim:2}] 
By  \eqref{eq:Pt}, the expression
$\sum_{0 \leqslant j \leqslant 1 < i \leqslant m} a_{i,j} y_j P_h^{(i)},$ for  $1 \leqslant h \leqslant s$, reads 
\begin{equation*}\label{eq:aiutoaiuto}
\begin{split}
& a_{2,0} \left( \sum_{\mu_0+\mu_1 = d_h-1} c^{(2)}_{h,\mu_0,\mu_1} y_0^{\mu_0+1} y_1^{\mu_1} \right) +
a_{2,1} \left( \sum_{\mu_0+\mu_1 = d_h-1} c^{(2)}_{h,\mu_0,\mu_1} y_0^{\mu_0} y_1^{\mu_1+1} \right) + \cdots + \\
& + \cdots + a_{m,0} \left( \sum_{\mu_0+\mu_1 = d_h-1} c^{(m)}_{h,\mu_0,\mu_1} y_0^{\mu_0+1} y_1^{\mu_1} \right) 
+ a_{m,1} \left( \sum_{\mu_0+\mu_1 = d_h-1} c^{(m)}_{h,\mu_0,\mu_1} y_0^{\mu_0} y_1^{\mu_1+1} \right).
\end{split}
\end{equation*}
By equating to $0$ the coefficients of $y_0^{d_h-k}y_1^k$, for $0 \leqslant k \leqslant d_h$, we find a homogeneous linear system
of $\sum_{h=1}^s d_h + s$ equations in the $2(m-1)$ variables $a_{i,j}$ and coefficients $c^{(i)}_{h,\mu_0,\mu_1}$, where $2 \leqslant i \leqslant m$ and $0 \leqslant j \leqslant 1$.
By \eqref{eq:linsys}, the map $\sigma$ is injective if and only if this system admits only the trivial solution. 
One  checks that this is  the case for a general choice of the coefficients $c^{(i)}_{h,\mu_0,\mu_1}$ and because of the assumption $t>0$ equivalent to $\sum_{i=1}^s d_i > 2(m-1) - s$.
Thus we deduce from \eqref{eq:normalbundle} that $h^0(\ell, N_{\ell/{Y_u}}) =0$.
\end{proof}
\end{claim}

By the irreducibility of $J$ and Claim \ref{claim:2}, for $u \in W_{\underline{d}}$ general, the Fano scheme $F(Y_u)$ is zero--dimensional, i.e. $Y_u$ contains finitely many lines. 
In particular, \eqref{codimT} yieds that $\codim_{S_{\underline{d}}}(W_{\underline{d}}) = t$ as desired. 

Finally, to show that $Y_u$ contains only one line for $u \in W_{\underline{d}}$ general, one makes a count of parameters, left to the reader, similar to the one in \eqref {codimT}, which shows that the codimension in $S^*_{\underline d}$ of the locus of $u$ such that $Y_u$  contains at least two lines is strictly larger than $t$.
\end{proof}
\end{proposition}

\begin{remark}\label{remark:Quadrics}
The case of quadrics is not covered by the assertion of Proposition \ref{proposition:Predonzan+}.
However, it is well known that a general quadric of dimension $k\geqslant 2$ contains a $(2k-3)$--dimensional family of lines, so that---using the notation of Proposition \ref{proposition:Predonzan+}---a general complete intersection of type $\underline{d}=(1,\dots,1,2)$ contains a family of lines of dimension $\theta$ and $W_{\underline{d}}=S^*_{\underline{d}}$ (cf. e.g. \cite[Corollary 2.2]{Bo}). 
On the other hand, the locus parameterizing conics containing a line coincides with the locus of singular conics and it has codimension $1$ in $H^0\left(\mathbb P^ 2, \mathcal{O}_{\mathbb{P}^2}(2)\right)$. 
In particular, if $\underline{d}=(1,\dots,1,2)$ and the general point of $S^*_{\underline{d}}$ parameterizes a conic, then $W_{\underline{d}}\subset S^*_{\underline{d}}$ has codimension 1.
\end{remark}


\subsection{The upper bound for the covering gonality}\label{section:UpperBound}
Given a  hypersurface $X\subset \mathbb{P}^{n+1}$, a smooth point $x\in X$ and an integer $h\geqslant 3$,  we are concerned with the existence of lines in  $V_x^h$ which do not pass through $x$. 
More precisely, we want to consider the closed subset $X^h_1\subset X$ defined as
$$
X^ h_1:=\overline{\left\{ x\in X\left| \; X \text{ is smooth at }x \text{ and } V_x^h \text{ contains a line which does not pass through }   x \right.\right\}}.
$$
\begin{lemma}\label{lemma:Example}
Let $n,d,h\geqslant 3$ be integers such that $d\geqslant 2n$ and $h(h+1)= 4n$. 
Then there exist a hypersurface $X\subset \mathbb{P}^{n+1}$ of degree $d$ and two points $p,q\in X$ such that:
\begin{itemize}
  \item[(i)] $X$ is smooth at $p$ and $q$;
 \item[(ii)]  $p\not\in X^ h_1$;
  \item[(iii)] $q\in X^ h_1$ and  $V^h_{q}$ contains a line $\ell$ such that $q\not\in\ell$, $p\in \ell$ and $V^h_{q}$ is smooth along $\ell$.
\end{itemize}
\begin{proof} 
Let $[y_0,\dots,y_{n+1}]$ be the coordinates in $\mathbb{P}^{n+1}$ and consider a homogeneous polynomial 
$$
F(y_0,\ldots,y_{n+1})=\sum _{|I|=d} a_I y^I
$$ 
of degree $d$, where for a multi--index $I=(i_0,\dots,i_{n+1})$, we denote by $|I|$ its length and we set ${y}^I:=y_0^{i_0}\cdots y_{n+1}^{i_{n+1}}$. 
Note that $I$ varies among the points with integral coordinates in the $n$--simplex $\Delta_d\subset \mathbb{R}^{n+1}_{\geqslant 0}$ with vertices $Q_0:=(d,0,\dots,0),\dots,Q_{n+1}:=(0,\dots,0,d)$.
For $0\leqslant \alpha\leqslant d$, we denote by $\Delta_{\alpha}^{(0)}\subset \Delta_d$ the $n$--subsimplex with vertices
$$
Q_0=(d,0,\dots,0),\, (d-\alpha,\alpha,0,\dots,0),\ldots,(d-\alpha,0,\dots,0,\alpha).
$$
 
Let $q:=[1,0,\dots,0], p:=[0,1,\dots, 0] \in \mathbb{P}^{n+1}$ and assume that the hypersurface $X := X_F \subset \mathbb{P}^{n+1}$ passes through $q$ and $p$, so that $a_{(d,0,\dots,0)}=a_{(0,d,\dots,0)}=0$.
Then $V^h_{q}$ depends on the multi--indices $I$ belonging to $\Delta_{h-1}^{(0)}$. 
In fact,  $V^h_q$ has equations
$$
G_{k}(y_0,\ldots,y_{n+1}):=\sum_{|J|=k}{\frac{\partial^k F}{\partial y^J}}(q) y^J=0\quad \text{for}\quad 1\leqslant k\leqslant h-1 \quad \text{where}\quad \frac{\partial^k }{\partial y^J}:=\frac{\partial^k}{\partial y_0^{j_0}\cdots \partial y_{n+1}^{j_{n+1}}}.
$$
In particular
$$
 \frac{\partial^k F}{\partial y^J}=\sum_{|I|=d,\,|J|=k,\, I\geqslant J} \frac{I!}{(I-J)!} a_I y^{I-J},
 $$
where, as usual,  $I!:=i_0!i_1!\cdots i_{n+1}!$ and $I-J$ and $I\geqslant J$ are defined componentwise.
For any $I\not\in \Delta_{k}^{(0)}$, $J$ such that $|J|=k$ and $I\geqslant J$, the value of $y^{I-J}$ at $q$ is zero.
Therefore, the non--zero coefficients $a_I$ in $G_k$ are such that $I\in \Delta_{k}^{(0)}$ for all $1\leqslant k\leqslant h-1$. 
Hence, in particular, $I\in \Delta_{h-1}^{(0)}$. 
 
Similarly, the non--zero coefficients $a_I$ in the equations of $V^ h_p$  are such that $I$ is an integral point in the $n$--subsimplex $\Delta_{h-1}^{(1)}$ with vertices 
$$
Q_1=(0,d,0,\dots,0),\, (h-1,d-h+1,0,\dots,0),\ldots,(0,d-h+1,0,\dots,0,h-1).
$$

Choose the coefficients of $F$ so that $G_1=y_{n+1}$, hence $q$ is a smooth point for $X$. 
The cone $V_q^h$ lies in $T_qX = V(y_{n+1}) \cong \mathbb{P}^n$.
Moreover, setting $H:= V(y_0)$, the section of $V_q^h$ with $H$ is 
$\Lambda_q^h= V(y_0,y_{n+1}, G_2,\ldots, G_{h-1})\subset V(y_0,y_{n+1})\cong \mathbb{P}^{n-1}$. 
Next we apply Proposition \ref {proposition:Predonzan+} with $m=n-1$, $s=h-2$ and $d_i=i+1$. 
By \eqref{eq:inequality}, we have
\begin{equation*}
t=\sum_{i=1}^s d_i + s - 2(m-1)=\frac{h(h+1)}{2}-2n+1=1.
\end{equation*}
If $h\geqslant 4$, Proposition \ref{proposition:Predonzan+} ensures the existence of a locus $W_{(2,\ldots,h-1)}$ of codimension 1 in $S_{(2,\ldots,h-1)}$, whose general point is a complete intersection of type $(2,\ldots,h-1)$ in $\bP^{n-1}$ containing a line.
The same holds for $h=3=n$ by Remark \ref{remark:Quadrics}. 

Since $W_{(2,\ldots,h-1)}$ has codimension 1 in $S_{(2,\ldots,h-1)}$ we can choose the coefficients $a_I$, with $I\in \Delta_{h-1}^{(1)}$ and $I\neq (0,d,0,\dots, 0)$, general enough so that $X$ is smooth at $p$ and  $\Lambda^ h_p$ contains no line. 
Moreover, we can choose the coefficients $a_I$, with $I\in \Delta_{h-1}^{(0)}$ which we did not fix yet, so that $p\in V^ h_q$ and $V^ h_q$ contains a line through $p$ different from  the line $\langle p,q\rangle$ and it is smooth along this line. 
This can be done without altering the coefficients $a_I$, with $I\in \Delta_{h-1}^{(1)}$, we already chose, since $\Delta_{h-1}^{(0)}$ and $\Delta_{h-1}^{(1)}$ have no integral point in common. 
Indeed, $\Delta^{(0)}_{h-1}\cap\Delta^{(1)}_{h-1}$ is not empty if and only if $(h-1,d-h+1,0,\ldots,0)\in \Delta^{(0)}_{h-1}$, which implies $2(h-1)\geqslant d$. 
However, conditions $d\geqslant 2n$, $h(h+1)=4n$ and $h\geqslant 3$ yield $h^2< h(h+1)=4n\leqslant 2d$, so that $h<\sqrt{2d}$.
Hence $2(h-1)< d$ and $\Delta^{(0)}_{h-1}\cap\Delta^{(1)}_{h-1}= \emptyset$.
Thus we conclude that $p$ and $q$ satisfy (i)--(iii). 
\end{proof}
\end{lemma}

\begin{lemma}\label{lemma:Z} Let $X\subset \mathbb{P}^{n+1}$ be a general hypersurface of degree $d\geqslant 2n$ and let $h\geqslant 3$ be an integer such that $h(h+1)\leqslant 4n$. 
\begin{itemize}
\item[(i)]  If $h(h+1)< 4n$, then $X^ h_1=X$ and if $x\in X$ is a general point, then $V^ h_x$ is a cone over a general complete intersection $\Lambda_{x}^h$ of type $\underline{\delta}= (1,1, 2,\ldots,h-1)$ in $\mathbb{P}^{n+1}$, which contains a family of lines of dimension $\theta$ given by \eqref{eq:inequality};
\item[(ii)]  If $h(h+1)= 4n$, then all components  of $X^ h_1$ have dimension $n-1$ and there is an irreducible component $Z$ of $X^ h_1$ such that, if $x\in Z$ is a general point, then $V^ h_x$ is a cone over a complete intersection $\Lambda_{x}^h$ of type $\underline{\delta}= (1,1, 2,\ldots,h-1)$ in $\mathbb{P}^{n+1}$, which contains a line and is parameterized by a general point of the locus $W_{\underline{\delta}} \subset  S^*_{\underline{\delta}}$ defined in Proposition \ref{proposition:Predonzan+}. 
\end{itemize}
\end{lemma}
\begin{proof}
Let
$$
\mathcal X := \left\{(x,F) \left|  F\in S^*_d \text{ and } x\in X_F\right. \right\} \subset \mathbb{P}^{n+1}\times S^*_d 
$$
be the universal hypersurface of degree $d$, and let 
$$ 
\mathbb{P}^{n+1} \stackrel{\tau}{\longleftarrow} \mathcal X \stackrel{\sigma}{\longrightarrow} S^*_d 
$$
be the projections onto the two factors. 
Let $H \subset \mathbb{P}^{n+1}$ be a general hyperplane and let 
$$
U := \mathcal X \cap \left( (\mathbb{P}^{n+1} \setminus H)\times S^*_d \right).$$
Without loss of generality, we may assume $H=V(y_{n+1})$.
Then we set $\underline{\alpha} := (1, 2,\ldots,h-1)$ and we consider the rational map 
$$
\begin{array}{rrcl}
\varphi\colon &U & {\dasharrow} & S^*_{\underline{\alpha}}  \\
& (x,F) & \longrightarrow & \displaystyle \left(G_1(y_0,\dots,y_{n},0),\dots,G_{h-1}(y_0,\dots,y_{n},0)\right),
\end{array}
$$
where $G_k(y_0,\dots,y_{n+1})$ is a polynomial of degree $k$ defined as in \eqref{eq:Gj}, $1 \leqslant k \leqslant h-1$, and the collection of these $(h-1)$--polynomials  provides the defining equations of the cone $V^h_x$ at $x\in X_F$. 
Therefore, $\varphi(x,F)$ gives the defining equations of $\Lambda_x^h=V_x^h \cap H$ in $H\cong \mathbb{P}^n$ and, if we fix the coordinates of $x\in X_F$ and the equation of the tangent space $T_xX_F$, then $\varphi$ is just the map $\zeta$ in \eqref{eq:zeta}.
Thanks to Lemma \ref{lemma:LambdaGeneral}, the map $\varphi$ is dominant. 
As in Proposition \ref{proposition:Predonzan+}, we consider the locus $W_{\underline{\alpha}} \subseteq S^*_{\underline{\alpha}}$ which parameterizes complete intersections containing a line.

If $h(h+1) < 4n$, then $W_{\underline{\alpha}} = S^*_{\underline{\alpha}}$ and the general point of $S^*_{\underline{\alpha}}$ parameterizes a smooth complete intersection containing a $\theta$--dimensional family of lines. 
Indeed, if $h\geqslant 4$, this fact follows from Proposition \ref{proposition:Predonzan+} applied to $H \cong \mathbb{P}^n$, with $s= h-1$, $d_i = i$ and $1 \leqslant i \leqslant h-1$.
If instead $h=3$, then $n\geqslant 4$ and $V^3_x$ is a cone over a quadric of dimension at least 2, as in Remark \ref{remark:Quadrics}.   
Thus, in both cases the general $(x,F) \in U$, and hence the general $(x,F) \in \mathcal X$, is such that $\varphi(x,F)\in S^*_{\underline{\alpha}}$ parameterizes a general complete intersection $\Lambda_{x}^h$ of type $\underline{\alpha}$ in $H$ containing a family of lines of dimension $\theta$.
 Moreover, the generality of the hyperplane $H$ ensures that $\Lambda_{x}^h$ is a general complete intersection of type $\underline{\delta}= (1,1, 2,\ldots,h-1)$ in $\mathbb{P}^{n+1}$, as in (i).

If $h(h+1) = 4n$,  Proposition \ref{proposition:Predonzan+} (and Remark \ref{remark:Quadrics} if $h=n=3$) yields $\codim_{S_{\underline{\alpha}}} (W_{\underline{\alpha}}) =1$, so that $\varphi^{-1} (W_{\underline{\alpha}})$ has codimension $1$ in $U$. 
Therefore, proving (ii) is equivalent to show that $\sigma\left(\varphi^{-1} (W_{\underline{\alpha}})\right)$ is dense in $S^*_d$.
We point out that if $\sigma\left(\varphi^{-1} (W_{\underline{\alpha}})\right)\subset S^*_d$ had codimension 1, then for any $F\in \sigma\left(\varphi^{-1} (W_{\underline{\alpha}})\right)$ and any smooth point $x\in X_F$, we would have $(x,F)\in \varphi^{-1} (W_{\underline{\alpha}})$. 
However, Lemma \ref{lemma:Example} assures the existence of a polynomial $F\in S^*_d$ and two smooth points $p,q\in X_F$ such that $(p,F)\not\in \varphi^{-1} (W_{\underline{\alpha}})$ and $(q,F)\in\varphi^{-1} (W_{\underline{\alpha}})$ (where, setting $\ell$ as in Lemma \ref{lemma:Example}(iii), a line $L\subset \Lambda_q^h$ is obtained by intersecting $V_q^h$ and the plane $\langle q,\ell\rangle$). 
Thus $\sigma(\varphi^{-1} (W_{\underline{\delta}}))$ is dense in $S^*_d$ and (ii) follows.
\end{proof}

\begin{remark}\label{remark:non-degenerate}
We point out that if $h\geqslant 3$ is an integer such that $h(h+1)= 4n$ as in Lemma \ref{lemma:Z}(ii), then a general hypersurface $X\subset \mathbb{P}^{n+1}$ of degree $d\geqslant 2n$ admits two points $p,q\in X$ as in Lemma \ref{lemma:Example}, i.e. $p\not\in X_1^h$, whereas $q\in X^ h_1$ and $V^h_{q}$ contains a line $\ell$ such that $q\not\in\ell$, $p\in \ell$ and $V^h_{q}$ is smooth along $\ell$.\\
To see this fact, we set $q\in Z$ to be a general point of the component $Z\subset X_1^h$ in Lemma \ref{lemma:Z}(ii), so that $V_q^h$ contains a line $\ell$ not passing through $q$, along which $V_q^h$ is smooth.
Then, we choose $p$ in $X\cap \ell$ and, up to projective transformation, we may assume that $q=[1,0,\dots,0], p=[0,1,\dots, 0] \in \mathbb{P}^{n+1}$, as in the proof of Lemma \ref{lemma:Example}.
Using the notation and the argument therein, we have that the coefficients $a_I$ governing the geometry of the cone $V_p^h$---i.e. those with $I\in \Delta^{(1)}_{h-1}$---are not affected by the geometry of $V_q^h$ and, in particular, by the condition $p\in X\cap \ell$.
Therefore, the generality of $X$ yields that the coefficients $a_I$ with $I\in \Delta^{(1)}_{h-1}$ are sufficiently general for having $p\not\in X_1^h$.
\end{remark}

We can now prove the upper bound in Theorem \ref{theorem:Main}.

\begin{theorem}\label{theorem:UpperBound}
Let $X\subset \mathbb{P}^{n+1}$ be a very general hypersurface of degree $d\geqslant 2n$.
Given a general point $x\in X$, there exists a plane curve $C\subset X$ passing through $x$, and having a point of multiplicity at least $\left\lfloor\frac{\sqrt{16n+1}-1}{2}\right\rfloor$ and gonality $\gon(C)\leqslant d-\left\lfloor\frac{\sqrt{16n+1}-1}{2}\right\rfloor$. 
In particular,
\begin{equation*}
\covgon(X)\leqslant d-\left\lfloor\frac{\sqrt{16n+1}-1}{2}\right\rfloor.
\end{equation*}

\begin{proof}
The case $n=1$ is trivial, since the projection from any point of the smooth curve $X\subset \mathbb{P}^2$ is indeed a $(d-1)$--gonal map, while the case $n=2$ is covered by \cite[Corollary 1.8]{LP}. 
So we assume $n\geqslant 3$. We set
$$
h:=\left\lfloor\frac{\sqrt{16n+1}-1}{2}\right\rfloor
$$
and we note that this is the maximal integer such that $h(h+1)\leqslant 4n$. 
We will prove the existence of a family of plane curves covering $X$ and having gonality at most $d-h$.

Thanks to Lemma \ref{lemma:Z}, the hypersurface $X$ contains an irreducible component $Z$ of $X^ h_1$ such that for the general $z \in Z$, the cone $V_z^h$ contains some line $\ell$ not passing through $z$, and for each such line $\ell$ we can consider the plane $\pi_{z,\ell}:=\langle z,\ell \rangle\subset V_z^h$. 
To ease the notation we set $\pi:=\pi_{z,\ell}$.
Note that $X$ does not contain rational curves (see \cite {V, V2}).  
Hence $\pi\not\subset X$ and $\pi$ cuts out on $X$ a curve $C:=C_{z,\ell}$ passing through $z$, which is possibly reducible, but none of its component is rational (in particular, it is not a line). 
Since $\pi\subset V^h_{q}$, any line $L\subset \pi$ passing through $z$ is such that the intersection cycle  $L\cdot X$ is of the form $L\cdot X=h\,z+x_1+\dots+x_{d-h}$, for some $ x_1,\dots,x_{d-h}\in C$.
Therefore, if $\widetilde C$ is the normalization of $C$, the projection from $z$ induces a morphism $\widetilde C\longrightarrow \ell\cong \mathbb{P}^1$, which is non--constant on any component of $C$ and has degree at most $d-h$.
In particular, the gonality of any irreducible component $\Gamma$ of $\widetilde C$ satisfies $\gon(\Gamma)\leqslant d-h$.

Actually, it shall follow from Theorem \ref{theorem:LowerBound} and Lemma \ref{lemma:Equality} that for general $z\in Z$, any such a component $\Gamma$ satisfies $\gon(\Gamma)\geqslant d-h-1$, and hence the general $C:=C_{z,\ell}\subset \pi$ shall turn out to be an irreducible plane curve with gonality $\gon(C)\leqslant d-h$.
Moreover, any line $L\subset \pi$ passing through $z$ meets $C$ at $z$ with multiplicity at least $h$, so that $C$ has a singular point of multiplicity at least $h$ at $z$.

To finish the proof we need to show that the curves $C_{z,\ell}$ cover an open subset of  $X$.
If $h(h+1)<4n$, then $Z=X$ by Lemma \ref{lemma:Z}, and the curves $C_{z,\ell}$ cover $X$.

If  $h(h+1)=4n$, then $Z$ has codimension 1 in $X$ by Lemma \ref{lemma:Z}.
In this case, proving that the curves $C_{z,\ell}$ cover  $X$ is equivalent to prove that (the closure of) the family of planes of the form $\pi_{z,\ell}$ as above is \emph{non--degenerate}, i.e. it sweeps out the whole projective space $\mathbb{P}^{n+1}$.

According to Remark \ref{remark:non-degenerate}, the hypersurface $X$ admits two smooth points $p,q\in X$ such that $p\not\in X^ h_1$, $q\in Z$, and $p\in \pi_{q,L}$, where $L\subset V_q^h$ is a line through $p$ not containing $q$. 
Assume by contradiction that the family $\mathcal P$ of planes  $\pi_{z,\ell}$ is degenerate, and let $\Pi$ be the proper subvariety of $ \mathbb{P}^{n+1}$ which is the union of the planes of $\mathcal P$. To simplify the argument, we assume that $\mathcal P$, and hence $\Pi$, is irreducible (the general case can be treated similarly by replacing $\mathcal P$ with each of its irreducible components). 

Since $Z\subset \Pi$, then 
$$n\geqslant \dim(\Pi)\geqslant \dim\left(\Pi\cap X\right)\geqslant \dim(Z)= n-1.$$ 
Note that $\Pi$  is not contained in $X$ because
$X$ contains no plane.    
Therefore $\dim(\Pi)>\dim(Z)$, hence
$\dim(\Pi)=n$ and $\Pi$ and $X$ intersect along a pure $(n-1)$-dimensional variety containing $Z$ as a component, and also along some other irreducible component $Y$ passing through $p\in X\smallsetminus Z$. 

By \eqref{eq:proposition3.8} one has   
\begin{equation}\label{eq:Example1}
\covgon(Z)+\covgon(Y)\geqslant 2(d-n-1). 
\end{equation}
On the other hand, the intersection $\Pi\cap X$ is covered by the curves $C_{z,\ell}$, and the sum of the gonalities of their irreducible components is at most $d-h$.
Hence
\begin{equation}\label{eq:Example2}
\covgon(Z)+\covgon(Y)\leqslant d-h. 
\end{equation}
By combining \eqref{eq:Example1} and \eqref{eq:Example2}, we deduce $d-2n-2+h\leqslant 0$, which is impossible for $d\geqslant 2n$ and $h\geqslant 3$.
Thus we reach a contradiction and we conclude that $\Pi$ does coincide with $\mathbb{P}^{n+1}$, as wanted.
\end{proof}
\end{theorem}


\subsection{Covering families of curves with low gonality and high tangency cones}\label{section:CoveringFamiliesAndCones}

Let $X$ be an irreducible complex projective variety of dimension $n$.
\begin{definition} 
A \emph{covering family of} $c$\emph{--gonal curves} consists of a smooth family $\mathcal{C}\stackrel{\pi}{\longrightarrow} T$ of irreducible curves endowed with a dominant morphism $f\colon \mathcal{C}\longrightarrow X$ such that for general $t\in T$, the fibre  $C_t:=\pi^{-1}(t)$ is a smooth curve with gonality $\gon(C_t)=c$ and the restriction $f_{t}\colon C_t\longrightarrow X$ is birational onto its image. 
\end{definition}
The \emph{covering gonality} of $X$ is the least integer $c>0$ such that there exists a covering family of $c$--gonal curves.
According to \cite[Remark 1.5]{BDELU}, we may assume that both $T$ and $\mathcal{C}$ are smooth, with $\dim(T)=n-1$.
Furthermore, up to base change, we may consider a commutative diagram
\begin{equation*}\label{diagram:Gonality}
\xymatrix{\mathcal{C} \ar[dr]_-{\pi} \ar[r]^-\varphi & T\times \mathbb{P}^1 \ar[d]^-{\mathrm{pr_1}} \\ & T,\\}
\end{equation*}
where the restriction $\varphi_t\colon C_t \longrightarrow \{t\}\times\mathbb{P}^1  \cong \mathbb{P}^1$ is a $c$--gonal map. 
If $y\in \mathbb{P}^1$ is a general point, we set $\varphi_t^{-1}(t,y)=\left\{q_1,\ldots,q_c\right\}\subset C_t$.
We can argue as in \cite[Example 4.7]{B} to construct a correspondence $\Gamma\subset X\times \left(T\times \mathbb{P}^1\right)$ of degree $c$ with null trace (cf. \cite[Section 4]{B}).
Then \cite[Theorem 2.5]{BCD} implies the following.
\begin{proposition}\label{proposition:BCD}
Let $X\subset \mathbb{P}^{n+1}$ be a smooth hypersurface of degree $d\geqslant n+3$, and let $\mathcal{C}\stackrel{\pi}{\longrightarrow} T$ be a covering family of $c$--gonal curves,  as above. 
If $c\leqslant 2d-2n-3$, then  $f(q_1),\ldots,f(q_c)\in X$ are contained on a line $\ell_{(t,y)}\subset \mathbb{P}^{n+1}$.  
\end{proposition}

For a general point $x\in X$, there exists a line $\ell_{(t,y)}\subset \mathbb{P}^{n+1}$ passing through $x$.
Moreover, since $X$ is smooth of degree $d\geqslant n+3$, then $X$ is of general type, hence it is not covered by lines, so that $\ell_{(t,y)}$ meets $X$ along a $0$--dimensional scheme of length $d=\deg X$.  
As we vary $(t,y)\in T\times\mathbb{P}^1$, the lines $\ell_{(t,y)}$ describe a subvariety $B_0\subset \mathbb{G}(1,n+1)$ of dimension $n$.
By taking a desingularization $B\longrightarrow B_0$, we have a commutative diagram
\begin{equation}\label{diagram:Congruence}
\xymatrix{\mathfrak{P} \ar[d]_-\phi \ar[r] \ar@/^1pc/[rr]^-\mu &\mathfrak{P}_0 \ar[d] \ar[r] & \mathbb{P}^{n+1} \\ B \ar[r]& B_0 & \\}
\end{equation}
where $\mathfrak{P}\stackrel{\phi}{\longrightarrow} B$ is the $\mathbb{P}^1$--bundle obtained as the pullback of the universal $\mathbb P^ 1$--bundle on $\mathbb{G}(1,n+1)$, and $\mu\colon \mathfrak{P}\longrightarrow \mathbb{P}^{n+1}$ is the obvious morphism, which is clearly dominant.

Finally, by arguing similarly to \cite[Theorem C]{BDELU}, we prove the following result.

\begin{proposition}\label{proposition:Cone}
Let $X\subset \mathbb{P}^{n+1}$ be a very general hypersurface of degree $d\geqslant 2n+2$. 
Consider a covering family of $c$--gonal curves  as above with $c\leqslant d-3$. 
Then: 
\begin{itemize}
  \item[(i)] there exists a point $x_t\in f(C_t)$ such that $f(C_t) \subset V_{x_t}^{d-c}\cap X$;
  \item[(ii)] the $c$--gonal map $\varphi_t: C_t\longrightarrow \{t\}\times\mathbb{P}^1\cong\mathbb{P}^1$ is the composition of $f_t$ with the projection from $x_t$. 
\end{itemize}
In particular, the image of $f(C_t)$ under the projection from $x_t$ is a rational curve $R_t\subset \Lambda_{x_t}^{d-c}$.
\begin{proof}
Since $d\geqslant 2n+2$, we deduce that $c\leqslant d-3\leqslant 2d-2n-3$.
Hence by Proposition \ref{proposition:BCD} there exists a line $\ell_{(t,y)}$ containing $f\left(\varphi_t^{-1}(t,y)\right)=\left\{f(q_1),\ldots,f(q_c)\right\}\subset f(C_t)$. 
We will prove that there exists a point $x_t\in f (C_t)$ depending only on $t$ such that $\ell_{(t,y)}\cdot X = (d-c)x_t+f(q_1)+\dots+f(q_c)$.

We argue by contradiction and we assume that, as $(t,y)\in \{t\}\times\mathbb{P}^1$ varies, the $0$--cycle \linebreak $(\ell_{(t,y)}\cdot X)-f(q_1)-\dots-f(q_c)$ moves describing a (possibly reducible) curve $D_t\subset X$.
We note that $D_t$ is dominated by the curve $\displaystyle E_t:=\overline{\left\{\left.(x,y)\in D_t\times \mathbb{P}^1\right| x\in \ell_{(t,y)}\right\}}$ under the first projection, whereas the second projection has degree $d-c$. Thus any irreducible component of $E_t$ has gonality at most $d-c$, and hence any irreducible component $D_t^{\prime}$ of $D_t$ satisfies $\gon(D_t^{\prime})\leqslant d-c$.
Moreover, since $c\geqslant \covgon(X)\geqslant d-n$ by \eqref{eq:theoremA} and $d\geqslant 2n+2$, we deduce $\covgon(X) >d-c$, so that $\covgon(X)>\gon(D_t^{\prime})$. 
Thus the closure of the locus swept out by the curves $D_t$ is a proper subvariety of $X$.

Let $S$ be an irreducible component of such a subvariety and let $1\leqslant s\leqslant n-1$ be its dimension. 
By \eqref{eq:proposition3.8}, one has
\begin{equation}\label{equation:CovGonSlower}
\covgon(S)\geqslant d-2n+s. 
\end{equation}
On the other hand, let us consider the family $\mathfrak{P}\stackrel{\phi}{\longrightarrow} B$ in \eqref{diagram:Congruence}. 
By construction, the general line $\ell_{(t,y)}$ intersects $S$.
Moreover, since $s\leqslant n-1$ and $\dim (B)=n$, if $x\in S$ is a general point, there is a family of dimension $n-s>0$ of lines of the original family passing through $x$.
Let us denote by $R\subset \mathfrak{P}$ the ramification divisor of the generically finite morphism $\mu\colon \mathfrak{P}\longrightarrow \mathbb{P}^{n+1}$  in \eqref{diagram:Congruence}.
Thus there exists an irreducible component $Z$ of ${R}$ such that $\mu(Z)=S$ and the restriction $\phi_{|Z}\colon Z\longrightarrow B$ is dominant.

Setting $e:=\deg \phi_{|Z}$, we claim that $\covgon(S)\leqslant e$.
Indeed, if we vary $(t,y)\in \{t\}\times\mathbb{P}^1$, the lines $\ell_{(t,y)}$ describe a rational curve $Q_t\subset B$
and the inverse image $\phi^{-1}(Q_t)$ intersects $Z$ along a curve $G_t$ which dominates $D_t$ by construction.
Since $Q_t$ is rational, we deduce
\begin{equation}\label{equation:CovGonSupper1}
\displaystyle \covgon(S)\leqslant \gon (D_t^{\prime})\leqslant \deg\left(\phi_{|G_t}\colon G_t\longrightarrow Q_t\right)\leqslant e.
\end{equation}

Now, we recall that for general $\left[\ell\right]\in B$, the fibre $L:=\phi^{-1}\left(\left[\ell\right]\right)$ satisfies $\left(L\cdot R\right)=n$ (see e.g. \cite[Proposition 1]{D1}), and the contribution of $Z$ to this intersection product is $e\cdot 
\ord_Z(R)$, where $\ord_Z(R)$ is the multiplicity of $Z$ in $R$.
By \cite[Corollary A.6]{BDELU}, one has $\ord_Z(R)\geqslant n-s$.
Therefore $e(n-s)\leqslant n$, and  \eqref{equation:CovGonSupper1} yields
\begin{equation}\label{equation:CovGonSupper2}
\covgon(S)\leqslant \frac{n}{n-s}. 
\end{equation}

Finally, by pooling \eqref{equation:CovGonSlower}, \eqref{equation:CovGonSupper2} and the assumption $d\geqslant 2n+2$, we obtain
$$
s+2\leqslant \frac{n}{n-s}=\frac{s}{n-s}+1,
$$ 
which fails for $1\leqslant s\leqslant n-1$.
Hence we get a contradiction, so that for general $(t,y)\in \{t\}\times\mathbb{P}^1$, the $0$--cycle $(\ell_{(t,y)}\cdot X)-f(q_1)-\dots-f(q_c)$ is supported at a point $x_t\in X$, which depends only on $t$.

Therefore, denoting by $\Sigma_t$ the closure of the surface swept out by the lines $\ell_{(t,y)}$ with ${(t,y)\in \{t\}\times\mathbb{P}^1}$, the intersection $X\cap \Sigma_t$ is a curve entirely supported on $f(C_t)$, so that $x_t\in f(C_t)$. 
In particular, we have that $\ell_{(t,y)}\cdot X = (d-c)x_t+f(q_1)+\dots+f(q_c)$, i.e. $\ell_{(t,y)}$ is a line of the ruling of the cone $V^{d-c}_{x_t}$. Then (ii) and the final assertion of the proposition follow.   
\end{proof}
\end{proposition}


\section{The vector bundles approach}\label{section:UniversalFamilies}

Let us consider a smooth hypersurface $X_F \subset \mathbb{P}^{n+1}$ defined by the vanishing of a non--zero polynomial $F$ of degree $d$ and, for the sake of brevity, let us set $\mathbb{G}:=\mathbb{G}(1,n+1)$. For a positive integer $r$, we define the variety
\begin{equation} \label{eq:dd} \ddt_{r,F} := \left\{\left.(x, [\ell]) \in X_F \times \mathbb{G} \right|\ell \cdot X_F \geqslant r x \right\}\end{equation}
(by $\ell \cdot X_F$ we mean the intersection scheme of $\ell$ and $X_F$, which is of finite length unless $\ell$ is contained in $X_F$, in which case, by convention, $\ell \cdot X_F \geqslant r x$ for all $x\in \ell$ and for all non--negative integers $r$).  For $r \leqslant {\rm min}\{d,2n+1\}$, $\ddt_{r,F}$ turns out to be nonempty, smooth, irreducible of
dimension $2n+1 - r$ (cf. Lemma \ref{lem:P4.1}).
The main goal of this section is to find necessary  conditions on the integers $r,d,n,l$ for the existence of a uniruled subvariety $Y\subset \ddt_{r,F}$ of dimension $l \geqslant n$, for  $X_F$  very general (see Corollary \ref{corollary:RationalCurvesVeryGen}). 
To this aim, we will follow the argument of \cite{P}, relying on the approach of \cite{E, V}.


\subsection{Hypersurfaces and lines of high contact order}\label{section:Contact}

Given two integers $n,d\geqslant 2$,  we set 
$$
\mathbb{P} := \mathbb{P}^{n+1},\quad N+1 := \dim_{\mathbb{C}}(S_d)={d+n+1\choose d}.
$$
Using the same notation as in the proof of Lemma \ref{lemma:Z}, let $\mathcal{X}\subset \mathbb{P}\times S^*_d$ be the universal hypersurface over $S^*_d$, which has dimension $N+1+n$, and let
\[ S^*_d  \stackrel {\sigma} \longleftarrow  \mathcal{X} \stackrel {\tau}\longrightarrow  \mathbb{P} \]
be the two projections. Besides, we denote by $\mathcal{P} \subset \mathbb{P} \times \mathbb{G}$ the universal family of lines over $\mathbb{G}$, endowed with the two  projections
\[ \mathbb G  \stackrel {\pi_2} \longleftarrow  \mathcal{P} \stackrel {\pi_1}\longrightarrow  \mathbb{P}. \] 
The morphisms $\pi_1$ and $\pi_2$ make $\mathcal P$ a $\mathbb{P}^n$--bundle over $\mathbb{P}$, respectively,  a $\mathbb{P}^1$-bundle over $\mathbb G$; since ${\Pic} (\mathbb{P}) \cong \mathbb{Z} [\cO_{\bP}(1)]$ and 
${\Pic} (\mathbb{G}) \cong \mathbb{Z} [\cO_{\mathbb{G}}(1)]$, where $\cO_{\mathbb{G}}(1)$ is the \emph{Pl\"ucker line bundle} on $\mathbb G$, it follows 
that the Picard group ${\Pic}(\cP)$ is generated by the line bundles 
$$
L:= \pi_2^*(\cO_{\mathbb{G}}(1)) \quad \text{and} \quad H:= \pi_1^*(\cO_{\bP}(1))
$$
(cf. also \cite[p. 609]{V2}).

Let $1 \leqslant r \leqslant {\min} \{d, n+1\}$ be an integer and, similarly to \eqref{eq:dd}, consider the variety 
\begin{equation*}
\ddt_r := \left\{\left.(x, [\ell], F) \in \cP \times S^*_d \right| \ell \cdot X_F \geqslant r x \right\}
\end{equation*}
endowed with the two projections
$$ \mathcal P  \stackrel {\psi} \longleftarrow  {\tilde \Delta_r} \stackrel {\phi}\longrightarrow  S^*_d .$$
Since $r \leqslant d$, the map $\psi$ is surjective; indeed, for any $(x, [\ell]) \in \mathcal{P}$, there is the triple $(x, [\ell], F)\in \ddt_r$, where
$X_F$ is the hypersurface consisting of $d$ general hyperplanes through  $x \in \mathbb{P}$. For any $(x, [\ell]) \in \cP$, one has
$$
\psi^{-1}\left((x,[\ell])\right) \cong \left\{ \left. F \in S^*_d \right|X_F \cdot \ell \geqslant r x\right\} \cong H^0(\bP, \cI_{rx/\bP} (d)) \smallsetminus\{0\}.
$$
Moreover, the line bundle $\cO_{\bP} (d)$ is $(r-1)$--very ample as $d\geqslant r$ (recall that a line bundle $L$ over a projective variety $Y$ is said to be 
{\em $k$--very ample} if the natural restriction map $H^0(Y,L) \to H^0(Z, L \otimes \cO_Z)$ is surjective, for any zero--dimensional subscheme $Z \subset Y$ of length at most $k+1$).  Thus from the exact sequence
$$ 
0 \to \cI_{rx/\bP} (d) \to \cO_{\bP} (d) \to \cO_{rx} (d)  \to 0
$$
we deduce that $h^1(\cI_{rx/\bP} (d)) =0$.
Hence $\psi$ is smooth of relative dimension $N+1-r$, and each fiber is irreducible.
Then $\ddt_r$ is smooth, irreducible, of dimension
$$
\dim(\ddt_r) = \dim(\cP) + h^0( \cI_{rx/\bP} (d)) = 2n+2 + N-r.
$$

\begin{remark}\label{remark:Invarianza} It follows from the  definition that  $\ddt_r$ is invariant under the action of $\mathrm{GL}(n+2)$ on  $\mathcal{P} \times S^*_d$ defined as follows. 
Given any $g\in \mathrm{GL}(n+2)$ and any triple $(x,[\ell],F)\in \mathcal{P} \times S^*_d$, then
$$
g \cdot (x, [\ell], F) := \left(g(x), g(\ell), \left({g^{-1}}\right)^*(F)\right) \in \mathcal{P} \times S^*_d,
$$
where $g(\ell)$ denotes the line which is (projectively) equivalent to $\ell$ under $g$.
\end{remark}

We note that $\ddt_r$ is endowed with the natural map 
$\rho\colon \ddt_r \longrightarrow \cX$ given by $(x, [\ell], F) \longmapsto (x, F)$, which fits in the following commutative diagram
\begin{equation*}
\xymatrix{\ddt_r  \ar[d]_\psi  \ar@/^1.5 pc/[rr]^-\phi \ar[r]^\rho & \mathcal{X} \ar[d]^\tau \ar[r]^\sigma & S^*_d \\  \mathcal{P} \ar[d]_{\pi_2} \ar[r]^{\pi_1} &  \mathbb{P} & \\ \mathbb{G} & & \,\,\,\,\,\,\,\, }
\end{equation*}

Since $r \leqslant n+1$,  \cite[Lemma 4.1]{P} may be rephrased as follows.
\begin{lemma}\label{lem:P4.1}
The map $\rho$ is surjective. Furthermore, if $F\in S^*_d$ is a general polynomial, then
\begin{itemize}
\item[(a)] the subvariety
$$
\ddt_{r,F} := \phi^{-1}(F)\subset \ddt_r
$$
is smooth, irreducible, of dimension $2n+1-r$;
\item[(b)] the restriction
${\tau\circ\rho}_{\left|{\ddt_{r,F}}\right.} \colon  \ddt_{r,F} \longrightarrow \bP$ maps onto
$$
\Delta_{r,F} :=\left\{ p \in X_F \left| \; \exists\; [\ell]\in \mathbb{G} \text{ s.t. } \ell \cdot X_F \geqslant rp\right.\right\} \subseteq X_F.
$$
\end{itemize}
\end{lemma}


\subsection{The canonical bundle of $\ddt_{r,F}$}\label{section:Canonical}

Let $F \in S^*_d$ be general.
The restriction to $\ddt_{r,F}$ of the map $\psi\colon \ddt_r \twoheadrightarrow \cP$ is an isomorphism onto its image, realizing $\ddt_{r,F}$ as in \eqref {eq:dd}.

Let $\mathcal L$ be the universal rank $2$ quotient bundle on $\mathbb G$ ($\mathbb{G}$ is endowed with the universal quotient bundle $\mathcal L$ and with the map $(\mathbb{C}^{n+2})^{\vee} \otimes \cO_{\mathbb{G}} \twoheadrightarrow \mathcal L$ such that, for any $[\ell] \in \mathbb{G}$, $\mathcal L_{[\ell]}^{\vee}$ is the $2$ dimensional sub--vector space  of $\mathbb{C}^{n+2}$ determining the line $\ell \subset \mathbb{P}$, cf. e.g. \cite[Thm.4.3.2]{Ser}. As such, $\mathcal P = {\rm Proj} ({\rm Sym}(\mathcal L))$ is the universal line over $\mathbb{G}$ via the map $\pi_2$).   
For any positive integer $m$, we set $\cE_m := \Sym^m(\cL)$, which is a vector bundle on $\mathbb{G}$. Since its fiber over $[\ell] \in \mathbb{G}$ identifies with $H^0(\ell, \cO_{\ell}(m))$, the rank of $\cE_m$ is $m+1$ and $c_1(\cE_m) = \cO_{\bG} \big(\frac{m(m+1)}{2}\big)$. 
The vector bundle $\pi_2^*(\cE_m)$ on $\cP$ is of the same rank and its fiber over $(x, [\ell]) \in \mathcal P$ once again identifies with $H^0(\ell, \cO_{\ell}(m))$. 

For any integer $1 \leqslant r \leqslant d$ and for any $(x, [\ell]) \in \mathcal P$, consider the  divisor $r x$ on $\ell \cong \mathbb{P}^1$ and the associated line bundle 
 $\cO_{\ell} (-rx) \cong \cO_{\mathbb{P}^1} (-r)$. Let $\cO_{\ell}(d)(-rx)$ denote the line bundle $\cO_{\ell}(d) \otimes \cO_{\ell}(-rx) \cong \cO_{\mathbb{P}^1} (d-r)$. 
 This fits in the exact sequence 
$$ 0 \to \cO_{\ell}(d)(-rx) \to \cO_{\ell}(d) \to \cO_{rx} (d) \to 0;$$ passing to cohomology, one has an injection  $H^0(\ell, \cO_{\ell}(d)(-rx)) \hookrightarrow 
H^0(\ell, \cO_{\ell}(d))$. As $(x, [\ell])$ varies in $\mathcal P$, one deduces that the vector bundle $\pi_2^*(\cE_d)$ on $\cP$ contains a rank ($d+1-r$)--sub--vector bundle
$\cA_{d,r} \hookrightarrow \pi_2^*(\cE_d)$, whose fiber over $(x, [\ell]) \in \cP$ identifies with $H^0(\ell, \cO_{\ell}(d)(-rx))$.  

Consider the exact sequence
\begin{equation}\label{eq:Brd}
0 \to \cA_{d,r} \to \pi_2^*(\cE_d) \to \cB_{d,r} \to 0
\end{equation}
defining the quotient $\cB_{d,r}$ as  a rank $r$ vector bundle.
Arguing as in \cite[p. 263--264]{P} one sees that $\ddt_{r,F} \subset \cP$ is the vanishing locus of a global section of $\cB_{d,r}$. 
Thus the normal bundle $N_{\ddt_{r,F}/\cP}$ of $\ddt_{r,F}$ in $\cP$  satisfies
\begin{equation}\label{eq:normal}
N_{\ddt_{r,F}/\cP} \cong {\cB_{d,r|\ddt_{r,F}}}.
\end{equation}
Moreover, the smoothness of $\ddt_{r,F}$ and adjunction formula yield
\begin{equation}\label{eq:Kddt}
\omega_{\ddt_{r,F}}  = \omega_{\mathcal P} \otimes \cO_{\ddt_{r,F}} \otimes c_1(N_{\ddt_{r,F}/\cP}) \cong \omega_{\mathcal P} \otimes \cO_{\ddt_{r,F}} \otimes c_1(\cB_{d,r}) 
= \cO_{\ddt_{r,F}} \big( K_{\cP} \big) \otimes c_1(\cB_{d,r}) 
\end{equation}
where $K_{\cP}$ denotes a canonical divisor of $\mathcal P$, thus $\omega_{\mathcal P} = \cO_{\mathcal P} \big( K_{\cP} \big)$, so $\omega_{\mathcal P} \otimes \cO_{\ddt_{r,F}} = 
\cO_{\ddt_{r,F}} \big( K_{\cP} \big)$. Since ${\Pic}(\cP)=\mathbb Z[H,L]$, one must have $\omega_{\mathcal P} \cong H^{\otimes a} \otimes L^{\otimes b}$, for some integers $a,b$. Using adjunction formula and restricting to any $\pi_1$--fiber of $\mathcal P$ (which is a $\mathbb{P}^n$) one gets $b = -(n+1)$ whereas restricting to any $\pi_2$--fiber of $\mathcal P$ (which is a line) one gets 
$a = -2$, i.e. $\omega_{\mathcal P} \cong (H^{\vee})^{\otimes 2} \otimes (L^{\vee})^{\otimes (n+1)}$.  With an usual abuse of notation, 
identifying Cartier divisors with associated line bundles  and so additive notation with tensor products, we can  write 
$K_{\cP} = - 2 H - (n+1)L$, as done also in \cite[p.\;609]{V2}. In order to compute $c_1(\cB_{d,r})$, we first note that 
the natural isomorphisms 
$$\cO_{\ell}(r)(-rx) \cong \cO_{\ell} \cong \cO_{\mathbb{P}^1} \; \text{and} \; H^0(\ell, \cO_{\ell}(d)(-rx)) \cong H^0(\ell, \cO_{\ell}(d-r)) \otimes H^0(\ell, \cO_{\ell}(r)(-rx))$$
give rise, as $(x, [\ell])$ varies in $\mathcal P$, to the following isomorphism of vector bundles on $\mathcal P$
$$
\cA_{d,r} \cong \pi_2^*(\cE_{d-r}) \otimes \cA_{r,r},
$$
where $\cA_{r,r} = r (L-H) \in {\rm Pic}(\cP)$.
Whence we deduce
\begin{eqnarray*}
c_1( \cA_{d,r} ) & =   \frac{(d+r)(d-r+1)}{2} L - r (d-r+1) H,
\end{eqnarray*}
so that \eqref{eq:Brd} implies
\begin{eqnarray*}
c_1( \cB_{d,r} ) & =  \frac{r(r-1)}{2} L + r (d-r+1) H
\end{eqnarray*}
and by \eqref{eq:Kddt}, we  conclude that
\begin{equation}\label{eq:Kddt2}
\omega_{\ddt_{r,F}}  =  
\cO_{\ddt_{r,F}} \left( (r(d-r+1)-2) H + \left( \frac{r(r-1)}{2} - n-1 \right) L \right).
\end{equation}


\subsection{Global generation lemmas}\label{section:Global}

Let $F \in S^*_d$ be a general point, and consider the inclusion $\ddt_{r,F} \subset\ddt_r$.
For any integer
$1 \leqslant l \leqslant \dim(\ddt_{r,F}) = 2n+1-r$, one has
\begin{equation}\label{eq:tangbund}
\bigwedge^{2n+1-r-l} {T_{\ddt_r|{\ddt_{r,F}}}} \otimes \omega_{\ddt_{r,F}} \cong \; \stackrel{\vee}{\Omega}^{2n+1-r-l}_{{\ddt_r}|{\ddt_{r,F}}} \otimes \; \Omega^{2n+2+N-r}_{{\ddt_r}|{\ddt_{r,F}}}
\cong \; \Omega^{N+1+l}_{{\ddt_r}|{\ddt_{r,F}}}.
\end{equation}
Consider the exact sequence 
\begin{equation}\label{eq:MdGrass}
0 \longrightarrow M_d \longrightarrow S_d \otimes \cO_{\bG}\stackrel {\rm ev}\longrightarrow \cE_d \longrightarrow 0
\end{equation}
where 
\begin{eqnarray*}
{\rm ev}_{[\ell]}\colon S_d \otimes \cO_{\bG, [\ell]} & {\longrightarrow} & {\cE_d}_{|\ell} \cong H^0(\ell, \cO_{\ell}(d)) \\
(F, [\ell]) & \longmapsto & F_{|\ell}
\end{eqnarray*}
and  $M_d:=\ker({\rm ev})$.
The fiber of $M_d$ at $[\ell] \in \bG$ identifies with $H^0(\ell, \cI_{\ell/\bP} (d))$.

Next we consider the exact commutative diagram 
$$
\begin{array}{ccccccc}
 &    &   &0 & &   & \\
 &    &   &\downarrow & &   & \\
 &    &   & \fM_d & &   & \\
  &    &   &\downarrow & &   & \\
 
0 \to &       \fN_{d,r} &  \to  & S_d \otimes \cO_{\cP} & \to &  \cB_{d,r}  & \to 0 \\
      &                        &       & \downarrow            &     &  | |        &       \\
0 \to &           \cA_{d,r}    & \to   &  \pi_2^*(\cE_d)  & \to  & \cB_{d,r} &  \to 0   \\
     &                        &       & \downarrow            &     &          &       \\
		     &                        &       & 0           &     &       &
\end{array}
$$
where the central vertical column is obtained 
by pulling ${\rm ev}$ back to $\cP$  via  $\pi_2$, the bottom row is \eqref{eq:Brd},
$\fN_{d,r}$ and $ \fM_d $ are defined as the appropriate kernels, 
and
$$
{\rm rk} (\fM_d)  = N-d, \quad {\rm rk} (\fN_{d,r})  = N+1-r.
$$

By the Snake Lemma we have the exact commutative diagram

\begin{equation}\label{eq:diag(punto)F}
\begin{array}{ccccccccr}
    &                        &       & 0          &     &       0   &       & & \\
    &                        &       & \downarrow            &     &  \downarrow         &   & &        \\

    &                0        &  \to     & \fM_d           &  \to    &    \fN_{d,r}      &  \to  & \cA_{d,r} &  \to\;0        \\

    &                        &       & \downarrow            &     &      \downarrow    &        & & \\

 &        &  & S_d \otimes \cO_{\cP} & = &   S_d \otimes \cO_{\cP}   &   & & \\
      &                        &       & \downarrow            &     &  \downarrow       &  & &      \\
0 \to &           \cA_{d,r}    & \to   &  \pi_2^*(\cE_d)  & \to  & \cB_{d,r} &  \to     & 0 & \\
     &                        &       & \downarrow            &     &        \downarrow  &  & &       \\
		     &                        &       & 0           &     &        0 &       &  & 
\end{array}
\end{equation}

\begin{lemma}\label{lem:aiuto} With notation as above, one has:
\begin{itemize}
\item[(i)] $\fM_d = \pi_2^*(M_d)$;
\item[(ii)] $\bigwedge^t {\fM}_d \otimes tL$ is globally generated, for any integer $1 \leqslant t \leqslant {\rm rk}(\fM_d) = N-d$;
\item[(iii)] ${\fM_{d|\ddt_{r,F}}} \hookrightarrow {\fN_{d,r|\ddt_{r,F}}} \hookrightarrow {T_{\ddt_r|\ddt_{r,F}}}$.
\end{itemize}
\begin{proof} Looking at the left--most exact column in \eqref{eq:diag(punto)F},  assertion (i) follows by pulling-back via $\pi_2$ the exact sequence \eqref{eq:MdGrass}. 

As for (ii), by \cite[Proposition 2.2(ii)]{P}, the vector bundle $M_d \otimes \cO_{\bG} (1)$ is globally generated,
i.e. $H^0(\bG,  M_d \otimes \cO_{\bG} (1)) \otimes \cO_{\bG} \twoheadrightarrow M_d \otimes \cO_{\bG} (1)$. 
Applying $\pi_2^*$ to this surjection and using (i), we obtain an induced surjection 
\begin{equation}\label{eq:gg}
H^0( \bG, M_d \otimes \cO_{\bG} (1)) \otimes \cO_{\cP} \twoheadrightarrow {\fM}_d \otimes L.
\end{equation}
Similarly, one has $H^0 (\cP, {\fM}_d \otimes L) \cong H^0(\cP, \pi_2^*(M_d \otimes \cO_{\bG} (1)))$.
We have
$$
{\pi_{2}}_*( \pi_2^*(M_d \otimes \cO_{\bG} (1))\cong 
 M_d \otimes \cO_{\bG} (1) \otimes {\pi_{2}}_*(\cO_{\cP}) \cong M_d \otimes \cO_{\bG} (1)
$$ 
because ${\pi_{2}}_*(\cO_{\cP})\cong \cO_{\bG}$ since the $\pi_2$--fibers are lines. Thus
$$
H^0(\cP, {\fM}_d \otimes L) \cong H^0(\bG, M_d \otimes \cO_{\bG} (1)).
$$
This isomorphism and \eqref {eq:gg} yield that ${\fM}_d \otimes L$ is globally generated.
Then its exterior powers, namely $\bigwedge^t {\fM}_d \otimes tL$ for any  $1 \leqslant t \leqslant {\rm rk}(\fM_d) = N-d$, are globally generated too.

Finally we prove (iii).  
Recall that, for any $F \in S^*_d$, $\ddt_{r,F}$ identifies with its image in $\cP$ under $\psi$.
Accordingly, we identify $\psi^*(T_{\cP})_{|{\ddt_{r,F}}}$ with ${T_{\cP|{\pi(\ddt_{r,F})}}}$ and, by abusing notation, we  denote it by ${T_{\cP|{\ddt_{r,F}}}}$. 
From the inclusions of schemes 
$$
\ddt_{r,F} \subset \ddt_r \subset \cP \times S^*_d
$$ 
and the fact that $\ddt_{r,F}$ is a $\phi$--fiber, we obtain the following exact sequence
\begin{equation}\label{eq:C}
0 \to {T_{\ddt_r|{\ddt_{r,F}}}} \to {T_{\cP|{\ddt_{r,F}}}}  \oplus (S_d \otimes \cO_{\ddt_{r,F}} ) \to  N_{\ddt_{r,F}/\cP} \to 0.
\end{equation}
Restricting the right--most exact column in \eqref{eq:diag(punto)F} to $\ddt_{r,F}$, we get
\begin{equation}\label{eq:B}
0 \to {\fN_{d,r|{\ddt_{r,F}}}}    \to   S_d \otimes \cO_{\ddt_{r,F}}  \to   {\cB_{d,r|{\ddt_{r,F}}}}  \cong N_{\ddt_{r,F}/\cP}  \to 0,
\end{equation}
where the isomorphism on the right is \eqref{eq:normal}, whereas the injectivity on the left follows from $\mathcal{T}\mathrm{or}^1( \cB_{d,r}; \cO_{\ddt_{r,F}}) = 0$, since $\cB_{d,r}$ is locally free.
The sequences \eqref{eq:B} and \eqref{eq:C} fit in the exact commutative diagram
\begin{equation}\label{eq:*}
\begin{array}{cccccccr}
   &  0    &   &  0  &   &   & &   \\
	  &  \downarrow    &   &  \downarrow  &   &   & &   \\
	
	0 \to &  {\fN_{d,r|{\ddt_{r,F}}}}     & \to   &   S_d \otimes \cO_{\ddt_{r,F}} & \to  & N_{\ddt_{r,F}/\cP} &  \to & 0   \\
	
	  &  \downarrow    &   &  \downarrow  &   & ||  &   & \\

	0 \to &  {T_{\ddt_r|{\ddt_{r,F}}}}    & \to   & {T_{\cP|{\ddt_{r,F}}}}  \oplus ( S_d \otimes \cO_{\ddt_{r,F}} )  & \to  &N_{\ddt_{r,F}/\cP} &  \to & 0   \\

	  &  \downarrow    &   &  \downarrow  &   &   &  &  \\
	
 &   {T_{\cP|{\ddt_{r,F}}}}    & = & {T_{\cP|{\ddt_{r,F}}}} &  &  &  &  \\
	  &  \downarrow    &   &  \downarrow  &   &   & &   \\
	
	  &  0    &   &  0  &   &   &  &  .
	\end{array}
	\end{equation}
Then the left--most column in \eqref{eq:*} gives ${\fN_{d,r|{\ddt_{r,F}}}} \hookrightarrow {T_{\ddt_r|{\ddt_{r,F}}}}$.
Finally, by restricting the upper  sequence in \eqref{eq:diag(punto)F} to $\ddt_{r,F}$, we get the inclusion ${\fM_{d|{\ddt_{r,F}}}} \hookrightarrow {\fN_{d,r|{\ddt_{r,F}}}}$, proving (iii).
\end{proof}
\end{lemma}

By \eqref{eq:tangbund}, Lemma \ref{lem:aiuto}(iii) ensures that for any integer $1 \leqslant l \leqslant 2n+1-r$, there is an injection
\begin{equation}\label{eq:aiuto2}
\bigwedge^{2n+1-r-l} {\fM_{d|\ddt_{r,F}}} \otimes \omega_{\ddt_{r,F}} \hookrightarrow \bigwedge^{2n+1-r-l} {T_{\ddt_r|{\ddt_{r,F}}}} \otimes \omega_{\ddt_{r,F}} \cong \Omega^{N+1+l}_{{\ddt_r}|{\ddt_{r,F}}}.
\end{equation} 
Moreover, \eqref{eq:Kddt2} implies that
\small{
\begin{equation}\label{eq:aiuto3}
\bigwedge^{2n+1-r-l} {\fM_{d|{\ddt_{r,F}}}} \otimes \omega_{\ddt_{r,F}} \cong \bigwedge^{2n+1-r-l} {\fM_{d|{\ddt_{r,F}}}} \otimes \left(\frac{r(r-1)}{2} - n -1\right) L_{|{\ddt_{r,F}}} \otimes (r (d-r+1)-2)  H_{|{\ddt_{r,F}}}.
\end{equation}
}

Since $H$ is globally generated, $1 \leqslant r \leqslant {\rm min} \{d, n+1\}$ and $d \geqslant2$, then $(r (d-r+1)-2)  H_{|{\ddt_{r,F}}}$ is globally generated, as well.
Similarly, by Lemma \ref{lem:aiuto}(ii), $\bigwedge^{2n+1-r-l} {\fM_{d|{\ddt_{r,F}}}} \otimes \left(\frac{r(r-1)}{2} - n -1\right) L_{|{\ddt_{r,F}}}$ is globally generated if $\frac{r(r-1)}{2} - n -1 \geqslant 2n+1-r-l$, that is if
\begin{equation}\label{eq:numero}
\frac{r(r+1)}{2} - 3n - 2 + l \geqslant 0.
\end{equation}
Taking into account \eqref{eq:aiuto3}, we deduce the following
\begin{lemma}\label{lem:gg}
If \eqref{eq:numero} holds, then $\bigwedge^{2n+1-r-l} {\fM_d}_{|_{\ddt_{r,F}}} \otimes \omega_{\ddt_{r,F}} $ is globally generated.
\end{lemma}


\subsection{Subvarieties of $\ddt_{r,F}$ of positive geometric genus}\label{section:PositiveGenus} 

Let $U \subset S^*_d$ be the open dense subset parametrizing polynomials $F \in S^*_d$ such that  $X_F$ is smooth. To ease notation, we still denote by $\ddt_r$ the restriction of $\ddt_r\subset \cP \times S_d$ to $U$, and let
$$U \stackrel{\phi}{\longleftarrow}\ddt_r \subset \cP \times U$$
be the projection.

Next we suppose that, 
up to possibly shrinking $U$ and replacing it with a suitable \'etale cover, there is a diagram
$$
\xymatrix{ \mathcal{Y} \,\, \ar@{^{(}->}[r]\ar[dr]_{\phi_{|\mathcal{Y}}} & \ddt_r\ar[d]^{\phi} \\  & U, }
$$
where $\cY \subseteq \ddt_r$ is an integral scheme and $\phi_{|\cY}$ is flat of relative dimension $l \leqslant 2n+1-r = \dim(\ddt_{r,F})$.
Then
$$
\dim(\cY) = N+1+l
$$
and, for a general $F \in U$,  $Y_F:= \phi^{-1}_{|{\mathcal Y}}(F)\subset \ddt_{r,F}$ is irreducible of dimension $l$.

Then we consider the diagram
$$
\xymatrix{ \widetilde{\mathcal{Y}} \ar[r]^{\nu}\ar[dr]_{\widetilde{\phi}} & \mathcal{Y}\ar[d]^{\phi_{|\mathcal{Y}}} \\  & U, }
$$
where $\nu$ is a desingularization of $\mathcal{Y}$ and $\widetilde{\phi}$ is the induced map.
It follows from the smoothness of $\widetilde{\cY}$ that the restriction of $\nu$ to a general fiber $\widetilde{Y}_F:=\widetilde{\phi}^{-1}(F)$ is a desingularization of $Y_F$.
Following \cite[Section 2]{P}, we shall also assume $\mathcal Y$ to be invariant under the action of $\mathrm{GL}(n+2)$ on $\cP \times S_d$ as in Remark \ref{remark:Invarianza}. 

Let us consider the map $\iota\colon \widetilde{\mathcal{Y}}\longrightarrow \ddt_r$, obtained by composing $\nu$ with the inclusion $\mathcal{Y}\hookrightarrow \ddt_r$.
For any integer $l \leqslant 2n+1-r$, the map $\iota^*: \Omega^1_{\ddt_r} \longrightarrow \Omega^1_{\widetilde{\cY}}$ induces the map
$$
\Omega^{N+1+l}_{{\ddt_r}|{\ddt_{r,F}}}  \stackrel{\beta}{\longrightarrow} \Omega^{N+1+l}_{{\widetilde{\cY}}|{\widetilde{Y}_F}} \cong \omega_{\widetilde{Y}_F},
$$
(cf. \cite[Section 2.2(3)]{P}). Composing $\beta$ with the injection \eqref{eq:aiuto2}, we  obtain a map
\begin{equation}\label{eq:Pac5}
\bigwedge^{2n+1-r-l} {\fM_{d|{\ddt_{r,F}}}} \otimes \omega_{\ddt_{r,F}} \stackrel{\alpha}{\longrightarrow} \; \omega_{\widetilde{Y}_F}.
\end{equation}

\begin{lemma}\label{lem:Pac2.3} 
Let $r$, $l$ and $d$ be positive integers such that \eqref{eq:numero} and $d \geqslant {\rm max} \{2,r\}$ hold.
Let $\mathcal{Y} \subset \ddt_r$ be any integral subscheme as above (in particular it is invariant under the ${\rm GL}(n+2, \mathbb{C})$--action on $\mathcal{P} \times S_d$). 
Then, for $F \in U$ general, the map
$$H^0(\alpha): H^0\left(\ddt_{r,F}, \bigwedge^{2n+1-r-l} {\fM_{d|{\ddt_{r,F}}}} \otimes \omega_{\ddt_{r,F}}\right) \longrightarrow H^0(\widetilde{Y}_F, \omega_{\widetilde{Y}_F})$$
induced  by \eqref{eq:Pac5} is non--zero.
In particular, the geometric genus of $\widetilde{Y}_F$ satisfies $p_g(\widetilde{Y}_F) := h^0(\widetilde{Y}_F, \omega_{\widetilde{Y}_F}) >0$.
\end{lemma}
\begin{proof} 
The proof uses the same approach as in \cite[Proofs of Lemmas 2.1(i) and 2.3]{P}, which in turn
follows \cite{E,E2,V,V2}. 
For the reader's convenience we recall the argument but we will be brief. 

Consider the exact sequence
\begin{equation}\label{eq:fan}
0 \longrightarrow T_{\ddt_r}^{\rm vert}\longrightarrow T_{\ddt_r} \stackrel{d\psi}\longrightarrow \psi^*(T_{\cP}) \longrightarrow 0,
\end{equation}
which defines $T_{\ddt_r}^{\rm vert}:={\rm Ker}(d\psi)$.
With the usual identification of $\ddt_{r,F}$ with its projection to $\mathcal P$ via $\psi$ (cf. \S \ref{section:Canonical}), \eqref{eq:fan} yields the exact sequence
\begin{equation}\label{eq:A}
0 \to {T_{\ddt_r}^{\rm vert}}_{|{\ddt_{r,F}}}\to {T_{\ddt_r|{\ddt_{r,F}}}} \to {T_{\cP|{\ddt_{r,F}}}} \to 0,
\end{equation}
where the injectivity on the left follows from $\mathcal{T}\mathrm{or}^1\left(\psi^*( T_{\cP}); \cO_{\ddt_{r,F}} \right) = 0$.
By comparing \eqref{eq:A} with the left--most column in \eqref{eq:*}, we deduce that ${\fN_{d,r|{\ddt_{r,F}}}} \cong {T^{\rm vert}_{\ddt_r|{\ddt_{r,F}}}}$. 
This, \eqref{eq:tangbund} and Lemma \ref{lem:aiuto}(iii) yield
$$
\bigwedge^{c} {\fM_{d|{\ddt_{r,F}}}} \otimes \omega_{\ddt_{r,F}} \hookrightarrow \bigwedge^{c} {T_{\ddt_r|\ddt_{r,F}}^{\rm vert}} \otimes \omega_{\ddt_{r,F}} \hookrightarrow \bigwedge^{c} {T_{\ddt_r|{\ddt_{r,F}}}} \otimes \omega_{\ddt_{r,F}} \cong \Omega^{N+1+l}_{{\ddt_r}|{\ddt_{r,F}}} \quad \text{where }  c:= 2n+1-r-l.
$$

By the ${\rm GL}(n+2)$--invariance of $\mathcal Y$ we have also the exact sequence 
\begin{equation*}
0 \longrightarrow T_{\mathcal Y}^{\rm vert}\longrightarrow T_{\mathcal Y} \stackrel{d\psi}\longrightarrow \psi^*(T_{\cP}) \longrightarrow 0
\end{equation*} 
(cf. \cite[Remark 2.3(a)]{E2} and \cite[Lemma 2.1(i)]{P}), hence given a smooth point $\xi:= (x,[\ell],F) \in \mathcal Y$ we have
$$
{\codim}_{T_{\ddt_r,\xi}^{\rm vert}} (T_{\mathcal Y,  \xi}^{\rm vert} ) = {\rm codim}_{\ddt_r}(\mathcal Y) = c.
$$
Now $H^0(\ddt_{r,F}, \bigwedge^{c} {\fM_{d|{\ddt_{r,F}}}} \otimes \omega_{\ddt_{r,F}})$ can be considered as a space of global sections of a line bundle over the relative Grassmannian of $c$--codimensional subspaces of ${T_{\ddt_r|{\ddt_{r,F}}}^{\rm vert}} \otimes \omega_{\ddt_{r,F}}$. Then the global generation of $\bigwedge^{c} {\fM_{d|\ddt_{r,F}}} \otimes \omega_{\ddt_{r,F}}$, which holds by Lemma \ref{lem:gg}, 
implies that there exists a section
$$
s \in H^0\left(\ddt_{r,F}, \bigwedge^{c} {\fM_{d|\ddt_{r,F}}} \otimes \omega_{\ddt_{r,F}}\right) \subset H^0\left(\ddt_{r,F}, \Omega^{N+1+l}_{{\ddt_r}|{\ddt_{r,F}}}\right)$$ such that
$$
s(x,[\ell]) \notin {\rm Ann}(T_{\mathcal Y, \xi}^{\rm vert}).
$$
Since $\iota \colon\widetilde{\mathcal Y} \to \ddt_r$ is generically an immersion, we obtain  a non--zero element in $H^0(\widetilde{Y}_F, \omega_{\widetilde{Y}_F}) $.
\end{proof}

If $l\geqslant n$, then $\frac{r(r+1)}{2} \geqslant 2n+2$ implies \eqref{eq:numero}, and we have:

\begin{corollary}\label{corollary:RationalCurvesVeryGen} 
Let $r$, $l\geqslant n$ and $d$ be positive integers such that $d \geqslant {\rm max} \{2,r\}$. Suppose there is a $\mathcal Y$ as in Lemma \ref {lem:Pac2.3},  such that for $F \in U$ general one has
$p_g(\widetilde{Y}_F)=0$ (e.g. when $\widetilde{Y}_F$ is uniruled), then

\begin{equation}\label{eq:r}
\frac{r(r+1)}{2} \leqslant 2n+1.
\end{equation}
Moreover, if $F \in U$ is very general and if $\ddt_{r,F}$ has a $l$--dimensional irreducible subvariety $Y_F$ fitting in a ${\rm GL}(n+2,\mathbb C)$--invariant family, whose desingularization $\widetilde{Y}_F$ has $p_g(\widetilde{Y}_F)=0$, then \eqref {eq:r} holds.
\end{corollary}
\begin{proof} The first part follows by Lemma \ref{lem:Pac2.3}. 
As for the final assertion, suppose  \eqref {eq:r} does not hold. Then Lemma \ref{lem:Pac2.3} yields that the set of polynomials $F\in U$ as in the statement is the union of countably many closed subsets of $U$. 
Hence the assertion holds.
\end{proof}


\section{Covering gonality of very general hypersurfaces}\label{section:Proof}

In this section we conclude the proof of Theorem \ref{theorem:Main}.
To start, we prove the following.
\begin{theorem}\label{theorem:LowerBound}
Let $X\subset \mathbb{P}^{n+1}$ be a very general hypersurface of degree $d\geqslant 2n+2$.
Then 
\begin{equation*}
\covgon(X)\geqslant d-\left\lfloor\frac{\sqrt{16n+9}-1}{2}\right\rfloor.
\end{equation*}

\begin{proof}
For $n=1$ and $n=2$, the assertion holds as the covering gonality of $X$ is $d-1$ and $d-2$, respectively (cf. \cite[Teorema 3.14]{C2} and \cite[Corollary 1.8]{LP}).
So we assume hereafter that $n\geqslant 3$.

Let $F\in S^*_d$ be very general, set 
$c:=\covgon(X_F)$
and consider a covering family $\mathcal{C}\stackrel{\pi}{\longrightarrow}T$ of $c$--gonal curves as in \S \ref {section:CoveringFamiliesAndCones} from which we keep the notation. We will assume that the curves of $\mathcal C$ have minimal degree among all $c$--gonal curves covering $X_F$.

Since $n\geqslant 3$, Theorem \ref{theorem:UpperBound} yields $c\leqslant d-3$.
Thus Proposition \ref{proposition:Cone} applies;
for general $(t,y)\in T\times \mathbb{P}^{1}$, there exists a line $\ell_{(t,y)}\subset V_{x_t}^{d-c}$ such that 
$$
\ell_{(t,y)}\cdot X = (d-c)x_t+f(q_1)+\dots+f(q_c)
$$
where $q_1+\cdots +q_c$ is a divisor of a $g^ 1_c$ on $C_t$.  Then we  define the map
\begin{equation*}
\begin{array}{crcl}
\Psi_{\mathcal{C}}\colon & T \times \mathbb{P}^1 & \dashrightarrow & \ddt_{d-c,F} \\
& (t,y)    & \longmapsto &  \displaystyle \left(x_t, \left[\ell_{t,y}\right], F\right)  
\end{array}.
\end{equation*} 
The variety $Y_{F,\mathcal C}:= \overline{\Psi_{\mathcal{C}}\left( T \times \mathbb{P}^1\right)}\subset \ddt_{d-c,F}$ is covered by the rational curves $\overline{\Psi_{\mathcal{C}}\left( \left\{t\right\} \times \mathbb{P}^1\right)}$, with $t\in T$. 
As $(t,y)\in  T \times \mathbb{P}^1$ vary, the lines $\ell_{t,y}$ describe a $n$--dimensional subvariety of the Grassmannian $\mathbb{G}(1,n+1)$ (cf. \S \ref{section:CoveringFamiliesAndCones}), so that $\dim (Y_{F,\mathcal C})=n$.

Define 
$$
Y_F:= \bigcup_{\mathcal C} Y_{F,\mathcal C}\subset \ddt_{d-c,F}
$$ 
with $\mathcal C$ varying among the (finitely many maximal) families of $c$--gonal curves of minimal degree covering $X_F$. We may pretend $Y_F$ to be irreducible, otherwise we replace it with one of its irreducible components. One has 
$\dim(Y_F)\geqslant n$ and $Y_F$ is covered by rational curves. 
Clearly the last assertion of Corollary \ref {corollary:RationalCurvesVeryGen}
can  be applied to $Y_F$, with $r=d-c$, so
$$
\frac{(d-c)(d-c+1)}{2}\leqslant 2n+1\quad \text{hence}\quad d-c\leqslant \frac {\sqrt{16n+9}-1}2.
$$
Being $c$ an integer, we find the lower bound in the statement. 
\end{proof}
\end{theorem}

The conclusion of the proof of Theorem \ref{theorem:Main} is given by the following elementary computational Lemma, whose proof can be left to the reader. 
\begin{lemma}\label{lemma:Equality}
Let $n\in \mathbb{N}$.
Then 
$$
\left\lfloor\frac{\sqrt{16n+9}-1}{2}\right\rfloor
=
\left\{
\begin{array}{ll}
\displaystyle \left\lfloor\frac{\sqrt{16n+1}-1}{2}\right\rfloor +1 & \text{if }n\in\left\{\left.4\alpha^2+3\alpha, 4\alpha^2+5\alpha+1\right|\alpha\in \mathbb{N}\right\} \\
& \\
\displaystyle  \left\lfloor\frac{\sqrt{16n+1}-1}{2}\right\rfloor & \text{otherwise.}
\end{array}
\right.
$$
\end{lemma}

\section{Final remarks, open problems and speculations}\label{sec:spec}

\subsection{Are all curves computing the covering gonality planar?}\label{ssec:plane}
As we mentioned in the Introduction, an interesting problem is to characterize the curves computing the covering gonality of a very general hypersurface $X\subset \mathbb{P}^{n+1}$ of degree $d\geqslant 2n+1$, in particular, one may ask the following:

\begin{question}\label{qu0} 
If $X\subset \mathbb{P}^{n+1}$ is a very general hypersurface of degree $d\geqslant 2n+1$, are the curves computing the covering gonality plane curves?
\end{question} 

The proof of Theorem \ref{theorem:Main} shows that this question is related to the existence of rational curves on certain complete intersections. Specifically, if $c=\covgon(X)$, $h=d-c$ and $x\in X^h_1$ (see \S  \ref{section:UpperBound}) is a very general point, one is led to ask the following:

\begin {question}\label{qu1} 
Does $\Lambda_x^h\subset \mathbb{P}^{n-1}$ contain rational curves other than lines?
\end{question}

Recall that when $Y\subset \mathbb{P}^{m}$ is a very general hypersurface of degree $2m-3$, then all  rational curves  on $Y$ are lines, whereas there are no rational curves on very general hypersurfaces of degree $2m-2$ (cf. \cite{P, V}).
The integer $h=\left\lfloor\frac{\sqrt{16n+1}-1}{2}\right\rfloor$ is the largest such that the locus of complete intersections $Y\subset \mathbb{P}^{n-1}$ of type $(2,\dots,h-1)$ containing a line has codimension at most 1 in the parameter space. 
Thus it is natural to investigate whether the general $Y$ containing a line may contain  other rational curves, i.e. Question \ref {qu1}  arises very naturally in this context. 

A negative answer to Question \ref {qu1} is a necessary condition for  an affirmative answer to Question \ref {qu0}. However this condition is not sufficient, as the cases
 $n\leqslant 7$ show. 
 
 If $n=3$, then $h=3$ and $X^ 3_1$ is the locus of points $x\in X$ such that the conic $\Lambda_x^3$ is reducible, so that the answer to Question \ref{qu1} is negative. However, if $x\in X$ is general, then 
 $\Lambda_x^h$ is an irreducible conic and the intersection of the irreducible cone $V_x^h$ with $X$ is not a plane curve but it still has gonality $d-3$. Similar considerations for $n=4$, for which again $h=3$. 
 
If $n=5$, then $h=4$, and $X^ 4_1$  is the locus of points $x\in X$ such that $\Lambda_x^3$, a $K3$ surface of degree 6 in $\mathbb P^ 4$, contains a line. 
The general such a surface contains no other rational curve, so again the answer to Question \ref{qu1} is negative. However, if $x\in X$ is general, then 
 $\Lambda_x^h$ is a general $K3$ surface of degree 6 in $\mathbb P^ 4$, which contains infinitely many singular rational curves, so that again there are (infinitely many) covering families of $X$ consisting of $(d-4)$--gonal curves other than plane curves. Similar considerations for $n=6,7$, for which still $h=4$.  
 
Hence, to hope for an affirmative answer to Question \ref {qu0}  it is necessary to assume $n$ sufficiently large.  

\subsection{The connecting gonality}

Another birational invariant introduced in \cite{BDELU} is the \emph{connecting gonality},
 defined as
\begin{displaymath}
\conngon(Y):=\min\left\{c\in \mathbb{N}\left|
\begin{array}{l}
\text{Given two general points }y_1,y_2\in Y,\,\exists\text{ an irreducible}\\ \text{curve } C\subset Y  \text{ such that }y_1,y_2\in C \text{ and }\gon(C)=c
\end{array}\right.\right\}.
\end{displaymath}
Since having $\conngon(Y)=1$ is equivalent to being rationally connected, the connecting gonality can be thought as a measure of the failure of $Y$ to satisfy such a property.

When $X\subset \mathbb{P}^{n+1}$ is  a very general hypersurface of degree $d\geqslant 2n+2$, our approach may determine restrictions also to $\conngon(X)$. Indeed, one could argue as in \S \ref{section:Predonzan+} and \S \ref{section:UpperBound}, and look for complete intersections in $\mathbb{P}^{n-1}$ of type $(2,3,\dots,h-1)$ containing large dimensional families of lines (see e.g. \cite[Corollary 2.2]{Bo}).
Since the dimension of a family of curves computing $\conngon(X)$ must be at least $2n-2$ (see e.g. \cite[Section 2.1]{B2}), a naive computation suggests the following upper bound
\begin{equation}\label{eq:ConnGon}
\conngon(X)\leqslant d-\left\lfloor\frac{\sqrt{8n+9}-1}{2}\right\rfloor.
\end{equation}
On the other hand, the curves computing the connecting gonality of $X$ are still governed by Proposition \ref{proposition:Cone}.
Thus a lower bound on $\conngon(X)$ could be obtained by improving the argument of Section 3.

In   analogy with the covering gonality, one may naively conjecture that inequality \eqref{eq:ConnGon} is actually an equality. 
Then  \eqref{eq:CovGon} would imply that the difference between the covering and connecting gonality diverges as $n$ grows.
This would answer to a question raised in \cite[Section 4]{BDELU}.

\subsection{The irrationality degree}

Given an irreducible projective variety $X$ of dimension $n$, for any positive integer $k\leqslant n$ one may define the \emph{$k$--irrationality degree} of $X$ to be the birational invariant
\begin{displaymath}
{\rm irr}_k(X):=\min\left\{c\in \mathbb{N}\left|
\begin{array}{l}
\text{Given a general point }x\in X,\,\exists\text{ an irreducible}\\\text{subvariety } Z\subseteq X  \text{ of dimension $k$ such that } x\in Z\text{ and}\\\text{there is a rational dominant map } Z\dasharrow \mathbb P^ k  \text{ of degree } c
\end{array}\right.\right\}.
\end{displaymath}
If $k=n$, this is the \emph{irrationality degree} ${\rm irr}(X)$ of $X$ (see \cite {B} for references), whereas ${\rm irr}_1(X)=\covgon(X)$. Of course one has
\[
{\rm irr}(X)\geqslant {\rm irr}_{n-1}(X)\geqslant \dots \geqslant {\rm irr}_2(X)\geqslant \covgon(X).
\]
It would be interesting to study this string of inequalities for $X\subset \mathbb P^ {n+1}$ a very general hypersurface as above. 
It is likely that our methods  could be useful for that. We hope to come back to this  in the future. 

Interestingly enough, a relevant amount of the inequalities above must consist of equalities.
Indeed, \cite[Theorem A]{BDELU} ensures that  ${\rm irr}(X)=d-1$, so that Theorem \ref{theorem:Main} yields that ${\rm irr}(X)-\covgon(X)\leqslant \left\lfloor\frac{\sqrt{16n+1}-1}{2}\right\rfloor$, which asymptotically equals $2\sqrt{n}$.

\subsection{Gaps} Finally, given an irreducible projective variety $X$ of dimension $n$, for any positive integer $k\leqslant n$ one can consider the following numerical set
\begin{displaymath}
N_k(X):=\left\{c\in \mathbb{N}\left|
\begin{array}{l}
\text{Given a general point }x\in X,\,\exists\text{ an irreducible}\\\text{subvariety } Z\subseteq X  \text{ of dimension $k$ such that } x\in Z \text{ and}\\\text{there is a rational dominant map } Z\dasharrow \mathbb P^ k  \text{ of degree } c
\end{array}\right.\right\}.
\end{displaymath}
Of course
\[
N_n(X)\subseteq N_{n-1}(X)\subseteq \dots N_2(X)\subseteq N_1(X).
\]
A \emph{gap} for $N_k(X)$ is an integer $c\in \mathbb N$ such that $c\not\in N_k(X)$. 
It is not difficult to see that the set of gaps of $N_n(X)$ is bounded, so is the set of gaps of $N_k(X)$ for all $k\leqslant n$ (see e.g. \cite[Problem 4.6]{BDELU}). 
It would be interesting to study the sets of gaps of $N_k(X)$, for $X$ a very general hypersurface in $\mathbb P^ {n+1}$ as above.
In this direction, we note that \cite[Theorem 1.3]{LP} leads to relevant results in the case of surfaces in $\mathbb{P}^{3}$.


\section*{Acknowledgements}

We started this research during the Workshop \emph{Birational geometry of surfaces}, held at University of Rome ``Tor Vergata'' on January $11^{\text{th}}--15^{\text{th}}$, 2016. So we would like to thank the Organizers for their invitation and Concettina Galati, Margarida Melo, Duccio Sacchi for stimulating conversations.
We are also grateful to Thomas Dedieu, Pietro De Poi, Lawrence Ein, Robert Lazarsfeld, and Edoardo Sernesi for helpful discussions.
Finally, we would like to thank the anonymous referee for pointing out various inaccuracies in the first version of this work.


\begin{thebibliography}{00}
\bibitem{B} F. Bastianelli, On symmetric products of curves, \emph{Trans. Amer. Math. Soc.} \textbf{364} (2012), 2493--2519.
\bibitem{B2} F. Bastianelli, On irrationality of surfaces in {$\mathbb{P}^3$}, \emph{J. Algebra} \textbf{488} (2017), 349--361.
\bibitem{BCD} F. Bastianelli, R. Cortini, P. De Poi, The gonality theorem of Noether for hypersurfaces, \emph{J. Algebraic Geom.} \textbf{23} (2014), 313--339.
\bibitem{BDELU} F. Bastianelli, P. De Poi, L. Ein, R. Lazarsfeld, B. Ullery, Measures of irrationality for hypersurfaces of large degree, \emph{Compos. Math.} \textbf{153} (2017), 2368--2393.
\bibitem{Bo} C. Borcea, Deforming varieties of $k$--planes of projective complete intersections, \emph{Pacific J. Math.} \textbf{143} (1990), 25--36.
\bibitem{C} C. Ciliberto, Osservazioni su alcuni classici teoremi di unirazionalit\`a per ipersuperficie e complete intersezioni algebriche proiettive, \emph{Ricerche di Mat.} \textbf{29} (1980), 175--191.
\bibitem{C2} C. Ciliberto, Alcune applicazioni di un classico procedimento di Castelnuovo, \emph{Seminari di geometria, 1982--1983 (Bologna, 1982/1983)}, Univ. Stud. Bologna, Bologna, 1984, 17--43.
\bibitem{Cle} H. Clemens, Curves on generic hypersurfaces, \emph{Ann. Sci. École Norm. Sup.} \textbf{ 19} (1986), 629--636.
\bibitem{DeMa} O. Debarre, L. Manivel, Sur la variété des espaces linéaires contenus dans une intersection complète, \emph{Math. Ann.} \textbf{312} (1998), 549--574.
\bibitem{D1} P. De Poi, On first order congruences of lines of {$\mathbb{P}^4$} with a fundamental curve, \emph{Manuscripta Math.} \textbf{106} (2001), 101--116.
\bibitem{E} L. Ein, Subvarieties of generic complete intersections, \emph{Invent. Math.} \textbf{94} (1988), 163--169.
\bibitem{E2} L. Ein, Subvarieties of generic complete intersections II, \emph{Math. Ann.} \textbf{289} (1991), 465--471.
\bibitem{LP} A.F. Lopez, P. Pirola, On the curves through a general point of a smooth surface in $\mathbb{P}^3$, \emph{Math. Z.} \textbf{19} (1995), 93--106. 
\bibitem{Mi} C. Miyazaki, Remarks on $r$--planes in complete intersections, \emph{Tokyo J. Math.} \textbf{39} (2016), 459--467. 
\bibitem{Morin} U. Morin, Sull'insieme degli spazi lineari contenuti in una ipersuperficie algebrica, \emph{Atti Accad. Naz. Lincei, Rend., Cl. Sci. Fis. Mat. Nat.} \textbf{24} (1936), 188--190.
\bibitem{P} G. Pacienza, Rational curves on  general projective hypersurfaces, \emph{J. Algebraic Geom.} \textbf{12} (2003), 245--267.
\bibitem{P2} G. Pacienza, Subvarieties of general type on a general projective hypersurface, \emph{Trans. Amer. Math. Soc.} \textbf{356} (2004), 2649--2661.
\bibitem{Pre} A. Predonzan, Intorno agli $S_k$ giacenti sulla variet\`{a} intersezione completa di pi\`{u} forme, \emph{Atti Accad. Naz. Lincei. Rend. Cl. Sci. Fis. Mat. Nat.}  \textbf{5} (1948), 238--242.
\bibitem{Segre}  B. Segre, Intorno agli $S_k$ che appartengono alle forme generali di dato ordine,  \emph{Atti Accad. Naz. Lincei, Rend., Cl. Sci. Fis. Mat. Nat.} \textbf{4} (1948), 261--265.
\bibitem{Ser} E. Sernesi, Deformations of Algebraic Schemes,  Grundlehren der mathematischen Wissenschaften, \textbf{334}, Springer-Verlag, Berlin, 2006.
\bibitem{V} C. Voisin, On a conjecture of Clemens on rational curves on hypersurfaces, \emph{J. Differential Geom.} \textbf{44} (1996), 200--213.
\bibitem{V2} C. Voisin, A correction on ``On a conjecture of Clemens on rational curves on hypersurfaces'', \emph{J. Differential Geom.} \textbf{49} (1998), 601--611.
\end{thebibliography}
\end{document}